\documentclass[12pt,leqno]{amsart}

\usepackage{enumerate}
\usepackage{graphicx}
\usepackage{epsfig}
\usepackage{amssymb,amsmath,amsthm,amsfonts}

\usepackage[letterpaper, margin=1in]{geometry} 

\usepackage {latexsym}

\usepackage{bbm}

\allowdisplaybreaks

\newcommand{\nc}{\newcommand}
\nc{\les}{\lesssim}
\nc{\nit}{\noindent}
\nc{\nn}{\nonumber}
\nc{\D}{\partial}
\nc{\diff}[2]{\frac{d #1}{d #2}}
\nc{\diffn}[3]{\frac{d^{#3} #1}{d {#2}^{#3}}}
\nc{\pdiff}[2]{\frac{\partial #1}{\partial #2}}
\nc{\pdiffn}[3]{\frac{\partial^{#3} #1}{\partial{#2}^{#3}}}
\nc{\abs}[1] {\lvert #1 \rvert}
\nc{\cAc}{{\cal A}_c}
\nc{\cE}{{\cal E}}
\nc{\cF}{{\cal F}}
\nc{\cP}{{\cal P}}
\nc{\cV}{{\cal V}}
\nc{\cQ}{{\cal Q}}
\nc{\cGin}{{\cal G}_{\rm in}}
\nc{\cGout}{{\cal G}_{\rm out}}
\nc{\cO}{{\cal O}}
\nc{\Lav}{{\cal L}_{\rm av}}
\nc{\cL}{{\cal L}}
\nc{\cB}{{\cal B}}
\nc{\cZ}{{\cal Z}}
\nc{\cR}{{\cal R}}
\nc{\cT}{{\cal T}}
\nc{\cY}{{\cal Y}}
\nc{\cX}{{\cal X}}
\nc{\cXT}{{{\cal X}(T)}}
\nc{\cBT}{{{\cal B}(T)}}
\nc{\vD}{{\vec \mathcal{D}}}
\nc{\efield}{\mathcal{E}}
\nc{\mE}{\mathcal{E}}
\nc{\vE}{{\vec \efield}}
\nc{\vB}{{\vec \mathcal{B}}}
\nc{\vH}{{\vec \mathcal{H}}}
\nc{\F}{  \mathcal{F} }
\nc{\ty}{{\tilde y}}
\nc{\tu}{{\tilde u}}
\nc{\tV}{{\tilde V}}
\nc{\Pc}{{\bf P_c}}
\nc{\bx}{{\bf x}}
\nc{\bX}{{\bf X}}
\nc{\bXYZ}{{\bf XYZ}}
\nc{\bY}{{\bf Y}}
\nc{\bF}{{\bf F}}
\nc{\bS}{{\bf S}}
\nc{\dV}{{\delta V}}
\nc{\dE}{{\delta E}}
\nc{\TT}{{\Theta}}
\nc{\dPsi}{{\delta\Psi}}
\nc{\order}{{\cal O}}
\nc{\Rout}{R_{\rm out}}
\nc{\eplus}{e_+}
\nc{\eminus}{e_-}
\nc{\epm}{e_\pm}
\nc{\eps}{\varepsilon}
\nc{\vnabla}{{\vec\nabla}}
\nc{\G}{\Gamma}
\nc{\w}{\omega}
\nc{\mh}{h}
\nc{\mg}{g}
\nc{\vphi}{\varphi}
\nc{\tlambda}{\tilde\lambda}
\nc{\be}{\begin{equation}}
\nc{\ee}{\end{equation}}
\nc{\ba}{\begin{eqnarray}}
\nc{\ea}{\end{eqnarray}}

\nc{\g}{\gamma}
\nc{\ol}{\overline}

\newtheorem{theorem}{Theorem}[section]
\newtheorem{lemma}[theorem]{Lemma}
\newtheorem{prop}[theorem]{Proposition}
\newtheorem{corollary}[theorem]{Corollary}
\newtheorem{defin}[theorem]{Definition}
\newtheorem{rmk}[theorem]{Remark}

\nc{\pT}{\partial_T}
\nc{\pz}{\partial_z}
\nc{\pt}{\partial_t}
\nc{\la}{\langle}
\nc{\ra}{\rangle}
\nc{\infint}{\int_{-\infty}^{\infty}}
\nc{\halfwidth}{6.5cm}
\nc{\figwidth}{10cm}
\newcommand{\f}{\frac}

\nc{\nlayers}{L} \nc{\nsectors}{M}
\nc{\indicator}{\mathbf{1}}
\nc{\Rhole}{R_{\rm hole}}
\nc{\Rring}{R_{\rm ring}}
\nc{\neff}{n_{\rm eff}}
\nc{\Frem}{F_{\rm rem}}
\nc{\R}{\mathbb R}
\nc{\C}{\mathbb C}
\nc{\Z}{\mathbb Z}
\nc{\DD}{\Delta}
\nc{\cD}{\mathcal D}
\nc{\lnorm}{\left\|}
\nc{\rnorm}{\right\|}
\nc{\rnormp}{\right\|_{\ell^{p,\eps}}}
\nc{\rar}{\rightarrow}
\nc{\mR}{\mathcal R}
\nc{\oo}{\"o}   

\nc{\os}{\overset{o}}
\begin{document}

\begin{abstract}

	We study the fourth order Schr\"odinger operator $H=(-\Delta)^2+V$ for a short range potential in three space dimensions.  We provide a full classification of zero energy resonances and study the dynamic effect of each on the $L^1\to L^\infty$ dispersive bounds.  In all cases, we show that the natural $|t|^{-\frac34}$ decay rate may be attained, though for some resonances this requires subtracting off a finite rank term, which we construct and analyze.  The classification of these resonances, as well as their dynamical consequences differ from the Schr\"odinger operator $-\Delta+V$.

\end{abstract}

\title[Dispersive estimates for Fourth order Schr\"odinger]{On the Fourth order Schr\"odinger equation in three dimensions: dispersive estimates and zero energy resonances}

\author[Erdo\u{g}an, Green, Toprak]{M. Burak  Erdo\smash{\u{g}}an, William~R. Green, and Ebru Toprak}
\thanks{The  first  author  is  partially  supported  by  NSF  grant  DMS-1501041. The second author is supported by  Simons  Foundation  Grant  511825. }
 \address{Department of Mathematics \\
University of Illinois \\
Urbana, IL 61801, U.S.A.}
\email{berdogan@illinois.edu}
 \address{Department of Mathematics\\
Rose-Hulman Institute of Technology \\
Terre Haute, IN 47803, U.S.A.}
\email{green@rose-hulman.edu}
\address{Department of Mathematics \\
Rutgers University \\
Piscataway, NJ 08854, U.S.A.}
\email{et400@math.rutgers.edu}

\maketitle

\section{Introduction}

We consider the linear fourth order Schr\"odinger equation in three spatial dimensions
\begin{align*} 
 i \psi_t = H \psi, \,\,\, \psi(0,x)= f(x), \,\,\, H:=\Delta^2+ V, \, \, \, x\in \mathbb R^3.
\end{align*}
Variants of this equation were introduced by Karpman \cite{K} 
and Karpman and Shagalov \cite{KS} to account for   small fourth-order
dispersion in the propagation of  laser beams in a bulk medium with Kerr
nonlinearity, and may be used to model other ``high dispersion" models. Linear dispersive estimates
have recently been studied, \cite{fsy,GT4,FWY}, we continue this study to understand the structure
and effect of zero energy resonances on the dynamics of the solution operator in the three
dimensional case.

  Fourth order Schr\"odinger equations have been studied in various contexts. For example,  the stability and instability of solitary waves in a non-linear fourth order equation were considered in \cite{LS}. Well-posedness and scattering problems for various nonlinear fourth order equations have been studied by many authors, see for example  \cite{MXZ1, MXZ2,P, P1,CLB,CLB1}. We  note that time decay estimates  we consider in this paper may be used in the study of special solutions to non-linear equations. 
  
In the free case, see \cite{BKS},  the solution operator $e^{-it \Delta^2 }$ in $d$-dimensions preserves the $L^2$ norm and satisfies the following   dispersive estimate
\begin{align*}
\|e^{-it \Delta^2 } f \|_{L^{\infty}(\R^d)} \les |t|^{-\f{d} 4} \|f\|_{L^1(\R^d)}. 
\end{align*}
In this paper we study the dispersive estimates in three spatial dimensions  when there are obstructions at zero, i.e the distributional solutions to $H\psi =0$ with $\psi \in L^\infty(\mathbb R^3)$.  We provide a full classification of the zero energy obstructions as a finite dimensional space of eigenfunctions along with a ten-dimensional space of two distinct types of zero-energy resonances, see Section~\ref{sec:classification}.  As in the four dimensional case, \cite{GT4},  the zero energy obstructions in three dimensions have a more complicated structure than that of the Schr\"odinger operators $-\Delta+V$, \cite{JN,ES}.  Let $P_{ac}(H)$ be the projection onto the absolutely continuous spectrum of $H$ and $V(x)$ be a real-valued, polynomially decaying potential.   We  prove dispersive bounds of the form 
$$
\|e^{-it H }P_{ac}(H) f \|_{L^{\infty}} \les |t|^{-\gamma} \|f\|_{L^1},
$$
or a variant with spatial weights, for each type of zero energy obstruction where $\gamma$ depends on the type of resonance.   Such estimates can be used to study asymptotic stability of solitons for non-linear  equations. 

We introduce some notation to state our main results.  We let $\la \cdot \ra=(1+|\cdot|)^{\f12}$, and let $a-$ denote $a-\epsilon$ for a small, but fixed value of $\epsilon>0$.   We define the polynomially weighted $L^p$ spaces,
\begin{align*}
L^{p,\sigma}(\mathbb R^3):=\{ f\, : \, \la \cdot \ra^{\sigma} f\in L^p(\mathbb R^3) \} .
\end{align*} 

We provide a precise definition and characterization of resonances in Section~\ref{sec:classification} and Definition~\ref{def:restype} below. We characterize the resonances in terms of distributional solutions to $H\psi=0$.  Heuristically, if $|\psi(x)| \sim 1$ as $|x|\to \infty$, we have a resonance of the first kind.  If $|\psi(x)| \sim |x|^{-1}$ as $|x|\to \infty$ we have a resonance of the second kind, and if $|\psi(x)|\sim |x|^{-\f32-}$ we have a resonance of the third kind.  The classification of the resonances in the fourth order Schr\"odinger equation requires a more detailed, subtle analysis than in the Schr\"odinger equation since the lower order terms in the expansion of Birman-Schwinger operator interact each other, see expansions of $M(\lambda)$  in Lemma~\ref{lem:M_exp}. This causes complications in the classification of threshold obstructions which do not arise in the case of Schr\"odinger's equation or in the  four dimensional case, see \eqref{F def}, \eqref{T2 def}, and Section~\ref{sec:classification}.
Our main results are summarized in the theorem below.

\begin{theorem}\label{thm:main}
	
	Let $V$ be a real-valued potential satisfying $|V(x)|\les \la x\ra^{-\beta}$ be such that there are no embedded eigenvalues  in $[0,\infty)$ except possibly at zero.  Then,
	\begin{enumerate}[i)]
		
		\item If zero is regular, then if $\beta>5$,
		$$
			\| e^{-itH}P_{ac}(H)\|_{L^1\to L^\infty} \les |t|^{-\f34}.
		$$
		
		\item If there is a resonance of the first kind at zero, then if $\beta>7$,
		$$
		\| e^{-itH}P_{ac}(H)\|_{L^1\to L^\infty} \les |t|^{-\f34}.
		$$
		
		\item If there is a resonance of the second kind at zero, then if $\beta>11$,
		$$
		\| e^{-itH}P_{ac}(H)\|_{L^1\to L^\infty} \les \left\{ \begin{array}{ll} 
		|t|^{-\f34} & |t|\leq 1 \\
		|t|^{-\f14} & |t|>1
		\end{array}\right.
		$$
		Moreover, there is a time-dependent, finite-rank operator $F_t$ satisfying $\|F_t\|_{L^1\to L^\infty}\les \la t\ra^{-\f14}$ so that
		$$
			\| e^{-itH}P_{ac}(H)-F_t\|_{L^1\to L^\infty} \les |t|^{-\f34}.
		$$

		\item If there is a resonance of the third kind at zero, then if $\beta>15$,
		$$
		\| e^{-itH}P_{ac}(H)\|_{L^1\to L^\infty} \les \left\{ \begin{array}{ll} 
		|t|^{-\f34} & |t|\leq 1 \\
		|t|^{-\f14} & |t|>1
		\end{array}\right.
		$$
		Moreover, there is a time-dependent, finite-rank operator $G_t$ satisfying $\|G_t\|_{L^1\to L^\infty}\les \la t\ra^{-\f14}$ so that
		$$
		\| e^{-itH}P_{ac}(H)-G_t\|_{L^1\to L^\infty} \les |t|^{-\f12}.
		$$
		Furthermore, one can improve this time decay at the cost of spatial weights,
		$$
			\| e^{-itH}P_{ac}(H)-G_t\|_{L^{1,\f52}\to L^{\infty,-\f52}} \les |t|^{-\f34}.
		$$		
		
	\end{enumerate}
	
\end{theorem}

As in the two-dimensional Schr\"odinger equation and four-dimensional fourth order equation, we have a `mild' type of resonance which does not affect the natural $|t|^{-\f{d}4}$ decay rate.   
As in \cite{fsy,GT4,FWY}, we assume absence of positive eigenvalues. Under this assumption, a limiting absorption principle for $H$ was established, see \cite[Theorem~2.23]{fsy}, which we use to control the large energy portion of the evolution, which necessitates the larger bound as $t\to 0$.  The large energy is unaffected by the zero energy obstructions, and our main contribution is to control the small energy portion of the evolution in all possible cases,  which we show is bounded for all time and decays for large $|t|$.  

In general, $|t|^{-\f{d}2}$ decay rate for the Schr\"odinger evolution is affected by zero energy obstructions. 
In particular, the time decay for large $|t|$ is slower if there are obstructions at zero, see for example \cite{JSS, Yaj, Sc2, goldE, eg2, EGG, GG1,GG2}. It is natural to expect zero energy resonances to effect the time decay of the fourth order operator as well.  This has been studied only in dimensions $d>3$;   by Feng, Wu and Yao, \cite{FWY}, when $d>4$ as an operator between weighted $L^2$ spaces, and by the second and third authors when $d=4$, \cite{GT4}.  These works built on the work of Feng, Soffer and Yao in \cite{fsy} which considered the case when zero is regular.  This work in turn had its roots in Jensen and Kato's work \cite{JenKat}, and \cite{JSS} for $-\Delta+V$.  
 
The free linear fourth order Schr\"odinger equation is studied by Ben-Artzi, Koch, and Saut \cite{BKS}. They present sharp estimates on the derivatives of the kernel of the free operator, (including $ \Delta^2 \pm \Delta $).  This followed work of Ben-Artzi and Nemirovsky which considered rather general operators of the form $f(-\Delta)+V$ on weighted $L^2$ spaces.  Further generalized Schr\"odinger operators of the form  $(-\Delta)^{m} + V$ were studied  in \cite{DDY}, \cite{soffernew}.   See also the work of Agmon \cite{agmon} and Murata \cite{Mur,Mur1,Mur2}.  In particular, Murata's results for operators of the form $P(D)+V$ do not hold for the fourth order operator due to the degeneracy of $P(D)= \Delta^2$ at zero.

There are not many works considering the perturbed linear fourth order Schr\"odinger equation outside of the previously referenced recent works. There has been study of  special solutions for nonlinear equations, see for example \cite{Lev,P,P1,MXZ1,MXZ2,Dinh}. See \cite{Lev1,LS} for a study of decay estimates for the fourth order wave equation.

Our results follow from careful expansions of the resolvent operators $(H-z)^{-1}$.  We develop these expansions as perturbations of the free resolvent, for which, by using the second resolvent identity (see also \cite{fsy}),  we have the following representation:
\begin{align}  \label{RH_0 rep}
R (H_0; z):=( \Delta^2 - z)^{-1} = \frac{1}{2z^{\f12}} \Big( R _0(z^{\f12}) - R_0 (-z^{\f12}) \Big), \quad z\in \mathbb C\setminus[0,\infty).
\end{align}
Here $H_0=(-\Delta)^2$ and the $R_0$ is the Schr\"odinger resolvent 
$R _0(z^{\f12}):=(-\Delta-z^{\f12})^{-1} $.
Since $H_0$ is essentially self-adjoint and $\sigma_{ac}(\Delta^2)= [0,\infty)$, by Weyl's criterion  $\sigma_{ess}(H) = [0,\infty)$ for a sufficiently decaying potential.  Let $\lambda \in \R^{+}$, we define the limiting resolvent operators by
\begin{align}
&\label{RH_0 def}R^{\pm}(H_0; \lambda ) := R^{\pm}(H_0; \lambda \pm i0 )= \lim_{\epsilon \to 0^+}(\Delta^2 - ( \lambda  \pm i\epsilon))^{-1}, \\
& \label{Rv_0 def} R_V^{\pm}(\lambda ) := R_V^{\pm}(\lambda  \pm i0 )= \lim_{\epsilon \to 0^+}(H - ( \lambda \pm i\epsilon))^{-1}.
\end{align}

Note that using the representation \eqref{RH_0 rep} for $R (H_0;z)$ in   definition \eqref{RH_0 def} with $z=w^4$ for $w$ in the first quandrant of the complex plane, and taking limits as $w\to \lambda$ and $w\to i\lambda$ in the first quadrant, we obtain 
\be\label{eq:4th resolvdef}
R^{\pm}(H_0; \lambda^4)= \frac{1}{2 \lambda^2} \Big( R^{\pm}_0(\lambda^2) - R_0(-\lambda^2) \Big),\,\,\,\lambda>0.
\ee 
Note that $R_0(-\lambda^2 ): L^2 \rightarrow L^2$ since $-\Delta$ has nonnegative spectrum. Further, by Agmon's limiting absorption principle, \cite{agmon}, $R^{\pm}_0(\lambda^2)$ is well-defined between weighted $L^2$ spaces. Therefore, $R^{\pm}(H_0; \lambda^4)$ is also well-defined between these weighted spaces. This property is extended to $R_V^{\pm}(\lambda )$ in \cite{fsy}.

As usual, we use functional calculus and the Stone's formula to write 
\begin{align}
\label{stone}
 \ e^{-itH}    P_{ac}(H) f(x) = \frac1{2\pi i}  \int_0^{\infty} e^{-it\lambda}   [R_V^+(\lambda)-R_V^{-}(\lambda)]  f(x) d\lambda.
 \end{align}
Here the difference of the perturbed resolvents provides the spectral measure.   Our analysis in the three-dimensional case differs from the four dimensional case and previous works on the Schr\"odinger operator in several ways.  First, the  behavior of the free resolvents  in \eqref{eq:4th resolvdef} provides technical challenges in which various lower order terms in the expansions interact.  These interactions complicate the inversion process as the operators whose kernels we study and need to invert are now the difference of different operators in the resolvent expansions, see \eqref{F def} and \eqref{T2 def} below.  Such difficulties are new to this case, in the analysis of the Schr\"odinger resolvents, see \cite{JN}, on can iterate the expansion procedure by examining the kernel of a single operator at each step.  The techniques  developed here may also be of use in dimensions $d=1,2$ or other high dispersion equations.    Furthermore, the difference between the `+' and `-' resolvents in the Stone's formula, \eqref{stone}, which is crucial in the Schr\"odinger operators and the four-dimensional case, do not improve the analysis except in the most singular term in the case of a resonance of the third kind.  Further, the classification of resonances differs from the Schr\"odinger case in several key aspects as shown in Section~\ref{sec:classification} below.

The paper is organized as follows.  In Section~\ref{sec:notation} we provide definitions of the various notations we use to develop the operator expansions.  In Section~\ref{sec:free} we develop expansions for the free resolvent and establish the natural dispersive bound for the free operator.  In Section~\ref{sec:low energy} we develop expansions for the perturbed resolvent in a neighborhood of the threshold for each type of resonance that may occur.  In Section~\ref{sec:low disp} we utilize these expansions to prove the low energy version of Theorem~\ref{thm:main}.  In Section~\ref{sec:large} we prove the high energy version of Theorem~\ref{thm:main}.  Finally, in Section~\ref{sec:classification} we provide a classification of the spectral subspaces associated to the different types of zero-energy obstructions.

\section{Notation}\label{sec:notation}
For the convenience of the reader, we have gathered the notation and terminology we use throughout the paper.  

For an operator $\mE(\lambda) $, we write $\mE(\lambda)=O_1(\lambda^{-\alpha})$ if it's kernel $\mE(\lambda)(x,y)$   has the property
\be\label{O1lambda} \sup_{x,y\in\R^3,  \lambda>0}\big[\lambda^{\alpha}|\mE(\lambda)(x,y)|+\lambda^{\alpha+1} |\partial_\lambda\mE(\lambda)(x,y)|\big]<\infty.
\ee 
Similarly, we use the notation $\mE(\lambda)=O_1(\lambda^{-\alpha}g(x,y))$ if $\mE(\lambda)(x,y)$   satisfies
\be\label{O1lambdag}   |\mE(\lambda)(x,y)|+\lambda  |\partial_\lambda\mE(\lambda)(x,y)|\les \lambda^{-\alpha}g(x,y).
\ee
Recall the definition of the Hilbert-Schmidt
norm of an operator $K$ with kernel $K(x,y)$,
$$
	\| K\|_{HS}:=\bigg(\iint_{\R^{6}}
	|K(x,y)|^2\, dx\, dy
	\bigg)^{\f12}.
$$
We also recall the following terminology from \cite{Sc2,eg2}:
\begin{defin}
	We say an operator $T:L^2(\R^2) \to   L^2(\R^2)$ with kernel
	$T(\cdot,\cdot)$ is absolutely bounded if the operator with kernel
	$|T(\cdot,\cdot)|$ is bounded from $  L^2(\R^2)$ to $ L^2(\R^2)$. 	
\end{defin}
We note that Hilbert-Schmidt and finite-rank operators are absolutely bounded operators.

We will use the letter $\Gamma$ to denote a generic absolutely bounded operator. In addition,   $\Gamma_\theta$  denotes a $\lambda$ dependent absolutely bounded operator  satisfying 
\be\label{eq:Gamma def}
\big \||\Gamma_\theta|\big \|_{L^2\to L^2}+ \lambda \big \||\partial_\lambda \Gamma_\theta|\big \|_{L^2\to L^2} \les \lambda^{\theta},\quad \lambda>0.
\ee
The operator may vary depending on each occurrence and $\pm $ signs.  The use of this notation allows us to significantly streamline the resolvent expansions developed in Section~\ref{sec:low energy} as well as the proofs of the dispersive bounds in Section~\ref{sec:low disp}.

We use the smooth, even low energy cut-off $\chi$ defined by $\chi(\lambda)=1$ if $|\lambda|<\lambda_0\ll 1$ and $\chi(\lambda)=0$ when $|\lambda|>2\lambda_0$ for some sufficiently small constant $0<\lambda_0\ll1$.  In analyzing the high energy we utilize the complementary cut-off $\widetilde \chi(\lambda)=1-\chi(\lambda)$.

\section{The Free  Evolution}\label{sec:free}
 
In this section we obtain expansions for the free fourth order Schr\"odinger resolvent operators $R^{\pm}(H_0; \lambda^4)$, using the identity \eqref{RH_0 rep}  and the Bessel function representation of the Schr\"odinger free resolvents $R^{\pm}_0(\lambda^2)$. We use these expansions to establish  dispersive estimates for the free fourth order Schr\"odinger evolution, and throughout the remainder of the paper to study the spectral measure for the perturbed operator. 

Recall the expression of the free Schr\"odinger resolvents in dimension three, (see \cite{GS} for example) 
$$
R^{\pm}_0(\lambda^2) (x,y)= \frac{e^{\pm i \lambda|x-y|}}{4 \pi |x-y|} .
$$
Therefore, by  \eqref{eq:4th resolvdef}, 
\be\label{eq:R0lambda}
R^\pm (H_0, \lambda^4) (x,y)= \frac{1}{2 \lambda^2}  \Bigg( \frac{e^{\pm i \lambda|x-y|}}{4 \pi |x-y|}  -  \frac{e^{-\lambda|x-y|}}{4 \pi |x-y|}  \Bigg).
\ee
When, $\lambda |x-y|<1$, we have the following representation for the $R(H_0, \lambda^4)$
\begin{align} \label{eq:R0low}
R^\pm(H_0, \lambda^4) (x,y) = \frac{a^{\pm}}{\lambda}  + G_0  + a_1^{\pm}  \lambda G_1 + a_3^{\pm} \lambda^3 G_3+ \lambda^4 G_4 + O(\lambda^5 |x-y|^6) .
\end{align} 
Here 
\begin{align}\label{adef}
&a^{\pm}:= \frac{1\pm i}{8 \pi}, \quad a_1^{\pm}= \frac{1\mp i }{ 8 \pi \cdot (3!)}, \quad a_3^{\pm} =\frac {1\pm i }{8 \pi \cdot (5!)}, \quad G_0 (x,y) = - \frac{|x-y|}{8 \pi}, \\\label{Gdef}
& G_1 (x,y) = |x-y|^2, \quad G_3 (x,y) = |x-y|^4, \quad G_4 (x,y) = -\frac{ |x-y|^5}{4\pi \cdot (6!)} .
\end{align}
When $\lambda |x-y|\gtrsim 1$, the expansion remains valid.
Notice that   $G_0=(\Delta^2)^{-1}$.

The following lemma will be used repeatedly to obtain low energy dispersive estimates. 
\begin{lemma}\label{lem:t-34bound} Fix $0<\alpha<4$. Assume that $\mE(\lambda)=O_1(\lambda^{-\alpha})$ for $0<\lambda\les 1$, 
then we have the bound
	\be\label{eq:t-34bound}
		 \bigg|
		\int_{0}^\infty e^{it\lambda^4}  \chi(\lambda) \lambda^3  \mE(\lambda)\, d\lambda
		\bigg| \les \la t\ra^{-1+\f\alpha4}.
	\ee
\end{lemma}
\begin{proof}
  By the support condition  and since $\alpha<4$,  the integral is bounded.  Now, for $|t|>1$ we rewrite the integral in \eqref{eq:t-34bound} as
	$$
	 \int_{0}^{t^{-\f14}} e^{it\lambda^4} \lambda^3  \chi(\lambda)\mathcal E(\lambda)\, d\lambda+\int_{t^{-\f14}}^\infty e^{it\lambda^4} \lambda^3 \chi(\lambda)\mathcal E(\lambda)\, d\lambda:=I+II.
	$$
	We see that
	$$
		|I|\leq \int_0^{t^{-\f14}} \lambda^{3-\alpha}\, d\lambda \les t^{-1+\f\alpha4 }.
	$$
	For the second term, we use $\partial_\lambda e^{it\lambda^4}/(4it)=e^{it\lambda^4} \lambda^3$ to  integrate by parts once.
	$$
		|II|\les \frac{e^{it\lambda^4} \mathcal E(\lambda)}{4it}\bigg|_{t^{-\f14}} + \frac{1}{t} \int_{t^{-\f14}}^\infty|\mathcal E'(\lambda)|\, d\lambda\les t^{-1+\f\alpha4}+ \frac{1}{t}\int_{t^{-\f14}}^\infty \lambda^{-\alpha-1}\, d\lambda\les t^{-1+\f\alpha4}.
	$$ 
\end{proof}

\begin{lemma}\label{lem:free bound}
	
	We have the bound
	$$
		\sup_{x,y\in \mathbb R^3} \bigg|
		\int_{0}^\infty e^{it\lambda^4} \chi(\lambda) \lambda^3  R^\pm(H_0,\lambda^4)(x,y)\, d\lambda
		\bigg| \les \la t\ra^{-\f34}.
	$$
	
\end{lemma}

\begin{proof}
	Note that the cancellation between $R^+$ and $R^-$ is not needed for, nor does it improve this bound. Using \eqref{eq:R0lambda} we have
	$$
|R^\pm (H_0, \lambda^4) (x,y)|=    \Bigg| \frac{e^{\pm i \lambda|x-y|}-e^{-\lambda|x-y|}}{8 \pi \lambda^2|x-y|}  \Bigg|\les \frac1{\lambda}
	$$
	uniformly in $x,y$ for $\lambda|x-y|>1$. For $\lambda|x-y|<1$, we have 
	$$
|R^\pm (H_0, \lambda^4) (x,y)|= \Bigg| \frac{e^{\pm i \lambda|x-y|}-1+1-e^{-\lambda|x-y|}}{8 \pi \lambda^2|x-y|}  \Bigg| \les \frac1{\lambda}
	$$
	by the mean value theorem. Similarly,
	$$|\partial_\lambda R^\pm (H_0, \lambda^4) (x,y)|\les \frac1{\lambda^2}
$$
uniformly in $x,y$. Therefore  
\be\label{eq:freeO1}
R^\pm (H_0, \lambda^4)=O_1(\lambda^{-1}),
\ee 
and the claim follows from Lemma~\ref{lem:t-34bound} with $\alpha=1$.	
\end{proof}
\begin{rmk} \label{rmk:large}  The   $t^{-\f34}$ bound is valid if we insert the high energy cutoff $\widetilde{\chi}(\lambda)=1-\chi(\lambda)$  in place of  the low energy cutoff $\chi(\lambda)$ in Lemma~\ref{lem:t-34bound}. However, the integral is not absolutely convergent, and is large for small $|t|$. That is,
$$
		 \bigg|
		\int_{0}^\infty e^{it\lambda^4}  \widetilde{\chi}(\lambda) \lambda^3  \mE(\lambda) \, d\lambda
		\bigg| \les | t |^{-1+\f\alpha4}.
$$
Consequently, we obtain the following estimate for the the free equation  
$$
\| e^{i t \Delta^2} f\|_{ L^{1} \rightarrow L^{\infty}} \les t^{-\f34}. 
$$ 
\end{rmk}
\section{Resolvent expansions near zero} \label{sec:low energy} 
In this section we provide the careful asymptotic expansions of the perturbed resolvent in a neighborhood of the threshold.  
To understand \eqref{stone} for small energies, i.e. $ \lambda \ll 1$, we use the symmetric resolvent identity.  We define $U(x)=$sign$(V(x))$, $v(x)=|V(x)|^{\f12}$, and write 
\begin{align} \label{resid}
R^{\pm}_V(\lambda^4)= R^{\pm}(H_0, \lambda^4) - R^{\pm}(H_0, \lambda^4)v (M^{\pm} (\lambda))^{-1} vR^{\pm}(H_0, \lambda^4) ,
\end{align}
where $M^{\pm}(\lambda) = U + v R^{\pm}(H_0, \lambda^4) v $. As a result, we need to obtain expansions for $(M^{\pm} (\lambda))^{-1}$.  The behavior of these operators as $\lambda \to 0$ depends on the type of resonances at zero energy, see Definition~\ref{def:restype} below. We determine these expansions case by case and establish their contribution to spectral measure in Stone's formula, \eqref{stone}.

Let $T:= U + v  G_0 v$, and recall \eqref{eq:Gamma def}, we have the following expansions. 
\begin{lemma}\label{lem:M_exp} For  $0<\lambda<1$ define  $M^{\pm}(\lambda) = U + v R^{\pm}(H_0, \lambda^4) v $. Let $P=v\langle \cdot, v\rangle \|V\|_1^{-1}$ denote the orthogonal projection onto the span of $v$.  We have
\begin{align}  
 \label{Mexp} M^{\pm}(\lambda)&= A^{\pm}(\lambda) +  M_0^\pm (\lambda),\\ 
\label{Apm}
A^{\pm}(\lambda)&= \frac{\|V\|_1 a^\pm}{\lambda} P+T,
\end{align} 
 where $T:=U+vG_0v$ and
 $M_0^\pm (\lambda)=\Gamma_\ell$, for any $0\leq \ell \leq  1$,  provided that $v(x)\les \la x \ra^{-\frac{5}{2}-\ell-}$.
Moreover, for each $N=1,2,...,$ and $\ell\in [0,1]$,
\be\label{M_0}
 M_0^\pm(\lambda)=\sum_{k=1}^N \lambda^k M_k^\pm +\Gamma_{N+\ell}
\ee
 provided that $v(x)\les \la x \ra^{-\frac{5}{2}-N-\ell-}$. Here the operators $M_k^\pm$ and the error term are  Hilbert-Schmidt, and hence absolutely bounded operators. In particular 
\begin{align}
  M^{\pm}_1 = a_1^{\pm} vG_1v,  \,\,\,\,\,\, M_2^\pm=0, \, \,\,\,\,\, M^{\pm}_3 =  a_3^{\pm}  vG_3v,  \,\,\,\,\,\,  M^{\pm}_4 =    vG_4 v,   \label{M134} 
\end{align} 
where, the $a^\pm_j$'s and $G_j$'s are defined in \eqref{adef} and \eqref{Gdef}.
\end{lemma} 	
\begin{proof} We  give a proof only for the case $N=1,2$, the other cases are similar.
Using the expansion \eqref{eq:R0low} for $\lambda|x-y|<1$ and \eqref{eq:R0lambda} for $\lambda|x-y|>1$, we have
$$
R^\pm (H_0, \lambda^4) (x,y)= \frac{a^{\pm}}{\lambda}  + G_0  + a_1^{\pm}  \lambda G_1+O_1(\lambda^3 |x-y|^4),\,\,\,\,\lambda|x-y|<1,
$$
\begin{multline*}
R^\pm (H_0, \lambda^4) (x,y)=\frac{a^{\pm}}{\lambda}  + G_0  + a_1^{\pm}  \lambda G_1+ \Big[\frac{e^{\pm i \lambda|x-y|}  -  e^{-\lambda|x-y|} }{8\pi \lambda^2|x-y|}   -
\frac{a^{\pm}}{\lambda}  - G_0  - a_1^{\pm}  \lambda G_1\Big]\\
= \frac{a^{\pm}}{\lambda}  + G_0  + a_1^{\pm}  \lambda G_1+  O_1(\lambda|x-y|^2),\,\,\,\,\lambda|x-y|>1.
\end{multline*} 
Using these in the definition of $M^{\pm}(\lambda)$ and $M^{\pm}_0(\lambda)$, we have 
$$
\Big|\big(M^{\pm}_0(\lambda)- a_1^{\pm} \lambda vG_1v\big)(x,y)\Big|\les v(x)v(y)|x-y|^{\ell+2}\lambda^{\ell+1},\,\,\,0\leq \ell\leq 2,
$$
$$
\Big|\partial_\lambda\big(M^{\pm}_0(\lambda)- a_1^{\pm} \lambda vG_1v\big)(x,y)\Big|\les v(x)v(y)|x-y|^{\ell+2}\lambda^{\ell },\,\,\,0\leq \ell\leq 2.
$$
This yields the claim for $N=1$ since the error term is an Hilbert-Schmidt operator if $v(x)\les \la x\ra^{-\frac52-1-\ell-}$. The case of $N=2$ also follows since $M_2=0$ and $\ell\in[0,2]$.
\end{proof}

The definition below classifies the type of resonances that may occur at the threshold energy. In Section~\ref{sec:classification}, we establish this classification in detail. Since the free resolvent is unbounded as $\lambda\to 0$, this definition is somehow analogous to the definition of resonances from \cite{JN} and \cite{Sc2} for the  two dimensional Schr\"odinger operators. However, there are important differences such as the appearance of the operators $T_1, T_2$ below.  Specifically, the lower order terms in the expansions interact in such a way that $T_1$ and $T_2$  are now the differences of two separate operators.  This phenomenon does not occur for the Schr\"odinger operators.
\begin{defin} \label{def:restype} 
	
	\begin{enumerate}[i)]
	
\item  Let $Q:=\mathbbm{1}-P$. We say that zero is regular point of the spectrum of $\Delta^2 +V$ provided $QTQ$ is invertible on $QL^2$. In that case we define $D_0:=(QTQ)^{-1}$ as an absolutely bounded operator on $QL^2$, see Lemma~\ref{lem:abs} below.  

\item  Assume that zero is not regular point of the spectrum. Let $S_1$ be the Riesz projection onto the kernel of $QTQ$. Then $QTQ+S_1$ is invertible on $QL^2$. Accordingly,  we
define $D_0 = (QT Q + S_1)^{-1}$,  as an operator on $QL^2$.  This doesn't conflict with the previous definition since $S_1=0$  when zero is regular. We say there is
a resonance of the first kind at zero if the operator 
\begin{align} \label{F def}
 T_1:=S_1TPTS_1- \frac{\|V\|_1}{3(8\pi)^2}  S_1vG_1vS_1  
 \end{align}  is invertible
on $S_1L^2$. 

\item  We say there is a resonance of the second kind if $T_1$ is not invertible on $S_1L^2$, but 
 \begin{align} \label{T2 def}
 T_2:= S_2vG_3vS_2+\frac{10}{3 }  S_2vWvS_2  
\end{align} 
 is invertible.  Here $S_2$ is the Riesz projection onto the kernel of $T_1$, and $W(x,y)=|x|^2|y|^2$.  Moreover, we define $D_1:= (T_1 +S_2)^{-1} $  as an operator on $S_1L^2$. 

\item  Finally if $T_2$ is not invertible we say there is a resonance of the third kind at zero. In this case the operator $T_3:= S_3 v G_4v S_3$  is always invertible on $S_3L^2$ where $S_3$ the Riesz projection onto the kernel of $T_2$, see Lemma~\ref{invertible}. 
We define $D_2:= (T_2 + S_3)^{-1}$ as an operator on $S_3L^2$.  

\end{enumerate}
\end{defin}

As in the four dimensional  operators, see the remarks after Definition~2.5 in \cite{EGG} and after Definition~3.2 in \cite{GT4}, $T$ is a compact perturbation of $U$.  Hence, the Fredholm alternative guarantees that $S_1$ is a finite-rank projection.
With these definitions first notice that, $ S_3 \leq S_2 \leq S_1 \leq Q $, hence all $S_j$ are finite-rank projections orthogonal to the span of $v$. Second, since $T$ is a self-adjoint operator and $S_1$ is the Riesz projection onto its kernel, we have 
$S_1 D_0= D_0 S_1 = S_1$.  Similarly, $S_2 D_1= D_1 S_2 = S_2$, $S_3 D_2= D_2 S_3 = S_3$.

\begin{lemma}\label{lem:abs} Let $|V(x)| \les \la x \ra^{-\beta}$ for some $\beta>5$, then $QD_0Q$ is absolutely bounded. 
\end{lemma}
\begin{proof} We prove the statement when $S_1 \neq 0$.  We first assume that $QUQ$ is invertible $QL^2 \rightarrow QL^2$. Using the resolvent identities, we have 
\begin{align*}
QD_0Q = QUQ - QD_0Q (S_1 + vG_0v) QUQ = QUQ - S_1UQ -QD_0Q vG_0v QUQ
\end{align*}

Note that $QUQ$ is absolutely bounded. Moreover, since $S_1$ is finite rank, any summand containing $S_1$ is finite rank, and hence absolutely bounded. For $QD_0Q vG_0v QUQ$, we note $vG_0v$ is an Hilbert-Schmidt operator for any $v(x) \les \la x \ra^{-5/2-}$ and $QD_0Q$ is bounded. Therefore,  $QD_0QvG_0v$ is Hilbert-Schmidt. Since the composition of absolutely bounded operators is absolutely bounded, $QD_0Q$ is absolutely bounded. 

If $QUQ$ is not invertible, one can define $\pi_0$ as the Riesz projection onto the kernel of $QUQ$ and see $QUQ + \pi_0$ is invertible on $QL^2$. Therefore, one can consider $ Q[ U + \pi_0 + S_1 + vG_0v -\pi_0]Q$ in the above argument to obtain the statement.  
\end{proof}

Our aim in the rest of this section is to prove Theorem~\ref{thm:Minvexp} below obtaining suitable expansions for $[M^\pm(\lambda) ]^{-1}$ valid as $\lambda \to 0$ under the assumption  that zero is regular and also in the cases when there are threshold obstructions. Recall the notation \eqref{eq:Gamma def} and that the operators $\Gamma_\theta$ vary from line to line.
\begin{theorem}\label{thm:Minvexp}  If zero is a regular point of the spectrum and if $|v(x)|\les \la x\ra^{-\frac52-}$, then
$$[M^\pm(\lambda) ]^{-1} =Q\Gamma_0Q+\Gamma_1.$$
If there is a resonance of the first kind at zero and if $|v(x)|\les \la x\ra^{-\frac72-}$, then
$$[M^\pm(\lambda) ]^{-1} =Q\Gamma_{-1}Q+Q\Gamma_0+\Gamma_0Q+\Gamma_1.$$
If there is a resonance of the second kind at zero and if $|v(x)|\les \la x\ra^{-\frac{11}2-}$, then
\begin{multline*}[M^\pm(\lambda) ]^{-1} =  S_2\Gamma_{-3} S_2+S_2\Gamma_{-2} Q+ Q\Gamma_{-2}S_2 + S_2\Gamma_{-1}+\Gamma_{-1}S_2 \\+Q  \Gamma_{-1}Q +Q  \Gamma_0+  \Gamma_0Q +\Gamma_1. \end{multline*} 
If there is a resonance of the third kind at zero and if $|v(x)|\les \la x\ra^{-\frac{15}2-}$, then
\begin{multline*}
[M^\pm(\lambda) ]^{-1} =\frac{1}{\lambda^4} S_3D_3S_3\\ +  S_2\Gamma_{-3} S_2+S_2\Gamma_{-2} Q+ Q\Gamma_{-2}S_2 + S_2\Gamma_{-1}+\Gamma_{-1}S_2  +Q  \Gamma_{-1}Q +Q  \Gamma_0+  \Gamma_0Q +\Gamma_1. 
\end{multline*} 
\end{theorem} 
Roughly speaking, modulo a finite rank term, the contribution to \eqref{stone} of all of the operators in these expansions are of the same size with respect to the spectral parameter $\lambda$.  We show in Lemma~\ref{lem:QvR} that in the contribution to \eqref{resid} having the operator $Q$ on one side allows us to gain a power of $\lambda$, while having $S_2$ allows us to gain two powers of $\lambda$ modulo the contribution of $G_0$.

Recall from \eqref{Mexp} that $M^\pm(\lambda) =   A^\pm(\lambda) + M_0^\pm(\lambda)$. 
 If zero is regular then we have the following expansion for $(A^\pm(\lambda))^{-1}$.
\begin{lemma}\label{regular} Let $ 0 < \lambda \ll 1$. If zero is regular point of the spectrum of $H$. Then, we have
\begin{align} \label{M inverse regular}
(A^\pm(\lambda))^{-1}= QD_0Q + g^{\pm}(\lambda)  S, 
\end{align}
where $g^{\pm}(\lambda)=(\frac{ a^\pm \|V\|_1}{\lambda} +c)^{-1} $ for some $c\in \R$, and
\begin{align} \label{def:S}
	S =\left[\begin{array}{cc}
	P& -PTQD_0Q\\ -QD_0QTP &QD_0QTPTQD_0Q
	\end{array} \right],
\end{align}
is a self-adjoint, finite rank operator.
 
Moreover, the same formula holds for $(A^\pm(\lambda)+S_1)^{-1}$ with $D_0=(Q(T+S_1)Q)^{-1}$ if   zero is not regular. 
 \end{lemma}
 \begin{proof} We prove the statement when $S_1\neq 0$. The proof is identical in the regular case. Recalling \eqref{Apm}, we write $A^\pm(\lambda)+S_1$ in the block format (using $PS_1=S_1P=0$):  
	\begin{align}
	A^\pm(\lambda)+S_1 = \left[\begin{array}{cc}
	\frac{ a^{\pm} \|V\|_1}{\lambda}  P + PTP & PTQ  \\ QTP &Q(T+S_1)Q 
	\end{array} \right]  :=\left[\begin{array}{cc}
	a_{11} & a_{12}  \\ a_{21} & a_{22}
	\end{array} \right]  
	\end{align}	
Since  $Q(T+S_1)Q$ is invertible, by Feshbach formula (see, e.g.,   Lemma~2.8 in \cite{eg2}) invertibility of  $A^\pm(\lambda)+S_1$ hinges upon the existence
    of $d=(a_{11}-a_{12}a_{22}^{-1}a_{21})^{-1}$. Denoting $D_0=(Q(T+S_1)Q)^{-1}:QL^2\to QL^2$, we have
	\begin{align*}
		d= \big(\tfrac{ a^{\pm} \|V\|_1}{\lambda} P+PTP-PTQD_0QTP\big)^{-1} =\big(\frac{ a^\pm \|V\|_1}{\lambda} +c\big)^{-1}P=:g^{\pm}(\lambda)P
	\end{align*}
	with $c = Tr(P T P - P T QD_0QT P) \in \R $.  
	Therefore,  $d$ exists if $\lambda$ is sufficiently small.
Thus, by the Feshbach formula,
	\begin{align}\label{feshbach}
		(A^\pm(\lambda)+S_1)^{-1}&=\left[\begin{array}{cc} d & -da_{12}a_{22}^{-1}\\
		-a_{22}^{-1}a_{21}d & a_{22}^{-1}a_{21}da_{12}a_{22}^{-1}+a_{22}^{-1}
		\end{array}\right]\\
		&=QD_0Q+ g^{\pm}(\lambda) S. \label{Ainverse}
	\end{align}
\end{proof}

Assume that $v(x)\les \la x\ra^{-\frac52-}$. Using \eqref{Mexp}, \eqref{M_0}, the resolvent identity and Lemma~\ref{regular} when zero is regular, we may write (for some $\epsilon>0$)
\begin{multline*}
[M^{\pm}]^{-1} = [A^\pm + M_0^\pm ]^{-1}=  [A^\pm + \Gamma_{\epsilon} ]^{-1}\\
=
[A^\pm ]^{-1}- [A^\pm ]^{-1}\Gamma_{\epsilon}[A^\pm ]^{-1}+ [A^\pm ]^{-1}\Gamma_{\epsilon}[M^{\pm}]^{-1}\Gamma_{\epsilon}[A^\pm ]^{-1}
=Q\Gamma_0Q+\Gamma_1,
\end{multline*}
proving Theorem~\ref{thm:Minvexp} in the regular case.

Assuming that $v(x)\les \la x\ra^{-\frac72-}$, by Lemma~\ref{lem:M_exp}, we have $M_0^\pm=\Gamma_1$. Also using \eqref{Mexp}  and Lemma~\ref{regular} we obtain the following expansion in the case  zero is not regular:
\begin{multline}\label{M+S1ex}
(M^{\pm} (\lambda)+S_1)^{-1} = (A^\pm(\lambda)+S_1 + M_0^{\pm}(\lambda))^{-1} \\= (A^\pm(\lambda)+S_1)^{-1} \sum_{k=0}^N (-1)^k [ M_0^{\pm}(\lambda) (A^\pm(\lambda)+S_1)^{-1}]^k+\Gamma_{N+1},\,\, N=0,1,\ldots,
\\ 
=QD_0Q+ \Gamma_1.
\end{multline}

The following lemma from \cite{JN} is the main tool to obtain the expansions of $ M^\pm(\lambda)^{-1}$ when zero is not regular.
\begin{lemma}\label{JNlemma}
	Let $M$ be a closed operator on a Hilbert space $\mathcal{H}$ and $S$ a projection. Suppose $M+S$ has a bounded
	inverse. Then $M$ has a bounded inverse if and only if
	$$
	B:=S-S(M+S)^{-1}S
	$$
	has a bounded inverse in $S\mathcal{H}$, and in this case
	\begin{align} \label{Minversegeneral}
	M^{-1}=(M+S)^{-1}+(M+S)^{-1}SB^{-1}S(M+S)^{-1}.
	\end{align}
\end{lemma}
We use this lemma with $M=M^{\pm}(\lambda)$ and   $S=S_1$.  Much of our technical work in the rest of this  section is devoted to finding appropriate expansions for the inverse of $B^\pm(\lambda)=S_1-S_1(M^{\pm}(\lambda)+S_1)^{-1}S_1$ on $S_1L^2$ under  various spectral assumptions.  For simplicity we work with $+$ signs and drop the superscript. 

We first list the orthogonality relations of various operators and projections we need. 
\begin{align}
\label{ort:1} &S_iD_j=D_jS_i=S_i, i>j,\\
\label{ort:2}&S_3\leq S_2\leq S_1\leq Q =P^\perp,\\
\label{ort:3}&S_1S= -S_1TP+S_1TPTQD_0Q,\,\,\,SS_1=-PTS_1+QD_0QTPTS_1,\\
\label{ort:4}&S_1SS_1=S_1TPTS_1,\\
\label{ort:5}&SS_2=S_2S=0,\\
\label{ort:6}&QM_1S_2=S_2M_1Q=S_3M_1=M_1S_3=0,\\
\label{ort:7}&S_2M_3S_3=S_3M_3S_2=0.
\end{align}
These can be checked using \eqref{def:S}, \eqref{M134}, and $Qv=S_2TP=S_2vG_1vQ=S_2x_jv=S_3x_ix_jv=0$, $i,j=1,2,3$ (see Lemmas~\ref{lem:S_2} and \ref{lem:S_3}  below).

Using \eqref{M_0} with $N=1$, $\ell=0+$ and \eqref{Ainverse} in \eqref{M+S1ex}, and then using \eqref{ort:4}, we obtain
\be \label{BlambdaS1}
B(\lambda)=S_1-S_1(M (\lambda)+S_1)^{-1}S_1 =- g(\lambda)S_1TPTS_1 + \lambda S_1M_1S_1 +\Gamma_{1+}, 
\ee
provided that $|v(x)|\les \la x\ra^{-\frac72-}$.

Using \eqref{M134}, we have 
\begin{multline}\label{T1calc}
g(\lambda)S_1TPTS_1 -\lambda S_1M_1S_1 = g(\lambda)[S_1TPTS_1 -a_1\frac{\lambda}{g(\lambda)} S_1vG_1vS_1]\\ =g(\lambda) T_1-ca_1\lambda g(\lambda) S_1vG_1vS_1=g(\lambda) T_1+\Gamma_2,
\end{multline}
where 
$$
T_1=S_1TPTS_1- \frac{\|V\|_1}{3(8\pi)^2}  S_1vG_1vS_1.
$$
The second equality follows from
$$
	g(\lambda)[S_1TPTS_1 -a_1\frac{\lambda}{g(\lambda)} S_1vG_1vS_1]
	=g(\lambda)[S_1TPTS_1 -a_1(a  \|V\|_1 +c\lambda) S_1vG_1vS_1],
$$
and recalling the definitions of $g(\lambda)$, \eqref{adef} and \eqref{M134} to see
$$
	a^\pm a_1^\pm =\frac{(\pm i+1)(\mp i+1)}{(8\pi)^2 (3!)}=\frac{2}{(8\pi)^2 (3!)}=\frac{1}{3(8\pi)^2.}
$$
In the case when there is a resonance of the first kind at zero, namely when $T_1$ is invertible,  using \eqref{T1calc} in \eqref{BlambdaS1}, we obtain
$$ 
B(\lambda)^{-1} =(-g(\lambda)T_1+\Gamma_{1+})^{-1}= \Gamma_{-1},
$$ 
provided that  $v(x)\les \la x\ra^{-\frac{7}2-}$.
Using this and \eqref{M+S1ex} in \eqref{Minversegeneral},   we obtain
\begin{multline*}
[M (\lambda)]^{-1}=QD_0Q+\Gamma_1+(QD_0Q+\Gamma_1)S_1 \Gamma_{-1}S_1(QD_0Q+\Gamma_1)\\ = Q\Gamma_{-1}Q+Q\Gamma_0+\Gamma_0Q+\Gamma_1,
\end{multline*}
proving Theorem~\ref{thm:Minvexp} in the case when there is a resonance of the first kind.

In the case when there is a resonance of the second or third kind, namely when $S_2\neq 0$, we need more detailed expansions for $B(\lambda)$, and hence for $(M (\lambda)+S_1)^{-1}$.

Using \eqref{M_0} and \eqref{M134} in \eqref{M+S1ex} we  obtain  
\begin{multline}\label{M+S1ex2}
(M (\lambda)+S_1)^{-1} = QD_0Q+g^\pm(\lambda)S -\lambda QD_0QM_1QD_0Q \\
-\lambda g^\pm(\lambda)\big[QD_0QM_1S+SM_1QD_0Q\big]+ \lambda^2 QD_0Q (M_1QD_0Q)^2 \\
-\lambda (g^\pm(\lambda))^2 SM_1S +\lambda^2g^\pm(\lambda) \big[S (M_1QD_0Q)^2+QD_0Q M_1SM_1QD_0Q +(QD_0Q M_1)^2S\big]\\
-\lambda^3 QD_0Q\big[M_3+M_1(QD_0QM_1)^2 \big]QD_0Q +\Gamma_{3+},
\end{multline}
provided that $v(x)\les \la x\ra^{-\frac{11}2-}$.

Using \eqref{ort:1}-\eqref{ort:4} and \eqref{T1calc} in \eqref{M+S1ex2}, we obtain
\begin{multline}\label{S1M+S1S1}
S_1(M (\lambda)+S_1)^{-1}S_1 = S_1+ g(\lambda) T_1 \\ -ca_1\lambda g(\lambda) S_1vG_1vS_1 
-\lambda g (\lambda)S_1\big[M_1S+SM_1\big]S_1+ \lambda^2 S_1 M_1QD_0QM_1S_1 \\
-\lambda (g (\lambda))^2 S_1SM_1SS_1 +\lambda^2g (\lambda) \big[S_1S M_1QD_0QM_1S+S_1 M_1SM_1S_1 +S_1M_1QD_0Q M_1SS_1\big]\\
-\lambda^3 S_1\big[M_3+M_1(QD_0QM_1)^2 \big]S_1 +\Gamma_{3+}.
\end{multline}
Therefore 
\begin{multline*}
B(\lambda)= S_1- S_1(M (\lambda)+S_1)^{-1}S_1 = -g(\lambda)T_1\\+  a_1\lambda g(\lambda) 
S_1(cM_1 + S M_1  +M_1S )S_1 
- \lambda^2  S_1M_1QD_0QM_1S_1\\
+\lambda (g (\lambda))^2 S_1SM_1SS_1 +\lambda^2g (\lambda) \big[S_1S M_1QD_0QM_1S+S_1 M_1SM_1S_1 +S_1M_1QD_0Q M_1SS_1\big]\\
+\lambda^3 S_1\big[M_3+M_1(QD_0QM_1)^2 \big]S_1 +\Gamma_{3+}.
\end{multline*}
 Let $U_1=S_1-S_2$, $U_2=S_2-S_3$, and $U=U_1+U_2$. In block form, we have
 \begin{align}\label{eq:BlambdaS3}
	B(\lambda) = \left[\begin{array}{cc} S_3B(\lambda) S_3
	 & S_3 B(\lambda) U  \\ UB(\lambda) S_3  &  U B(\lambda) U 
	\end{array} \right],
	\end{align}
	 \begin{align}\label{eq:Blambda}
	UB(\lambda)U = \left[\begin{array}{cc} U_2 B(\lambda) U_2
	 & U_2 B(\lambda) U_1  \\ U_1B(\lambda) U_2  &  U_1B(\lambda) U_1
	\end{array} \right].
	\end{align}
We first invert $UB(\lambda) U$ for small $\lambda$. We have 
$$
U_1B(\lambda) U_1=-g(\lambda) U_1T_1U_1+U_1\Gamma_2U_1,
$$
Using \eqref{ort:5} and \eqref{ort:6}, we obtain  
$$
U_1B(\lambda) U_2  = a_1\lambda g(\lambda) 
 U_1S vG_1v  U_2  +U_1\Gamma_3U_2=-a_1\lambda g(\lambda) 
 U_1TP vG_1v  U_2  +U_1\Gamma_3U_2,
$$
Similarly,
$$
U_2B(\lambda)U_1   =-a_1\lambda g(\lambda) 
  U_2  vG_1v PT U_1 +U_2\Gamma_3U_1,
$$
and
\begin{multline*}
U_2B(\lambda) U_2 = \lambda^3 U_2M_3U_2- \lambda^2g(\lambda)U_2M_1SM_1U_2+\Gamma_4 \\
=a_3\lambda^3 U_2vG_3vU_2-a_1^2\lambda^2g(\lambda)U_2vG_1vSvG_1vU_2+U_2\Gamma_{3+}U_2.
\end{multline*}
Note that by \eqref{ort:6}
$$
U_2vG_1vSvG_1vU_2= U_2vG_1vPvG_1vU_2= \|V\|_{L^1} U_2vWvU_2,
$$ 
where $W(x,y)=|x|^2|y|^2$. In the second equality we used $G_1(x,y)=|x|^2-2x\cdot y +|y|^2$ and $S_2 x_j v =S_2v=0$.
Also noting that 
$$
\frac{a_3\lambda}{ g(\lambda)}=a_3a\|V\|_1+ca_3\lambda=\frac{2i\|V\|_1}{5!(8\pi)^2}+ca_3\lambda,\,\,\,\,\,\,
a_1^2=\frac{-2i}{(3!)^2(8\pi)^2},
$$
we obtain
\begin{multline*}
U_2B(\lambda) U_2 =  \frac{2i }{5!(8\pi)^2} \|V\|_{L^1}\lambda^2g(\lambda) [    U_2vG_3vU_2+\frac{10}{3 }  U_2vWvU_2]+\Gamma_4\\
= \frac{2i }{5!(8\pi)^2} \|V\|_{L^1}\lambda^2g(\lambda) U_2T_2U_2+U_2\Gamma_{3+}U_2.
\end{multline*}

If   $U_1=0$, i.e. $S_1=S_2$, then we can invert $UBU$ as
$$
(UB(\lambda)U)^{-1}= \frac{5!(8\pi)^2} {2i \|V\|_{L^1}\lambda^2g(\lambda) } (U_2T_2U_2)^{-1}+U_2\Gamma_{-3+}U_2 = U_2  \Gamma_{-3}U_2.
$$

If  $U_1 \neq 0$, we invert $UB(\lambda)U$ using Feshbach's formula. 
Note that, we can rewrite \eqref{eq:Blambda} using the calculations above:   
$$
	UB(\lambda)U = -g(\lambda)\left[\begin{array}{cc} -\frac{2i }{5!(8\pi)^2} \|V\|_{L^1}\lambda^2  U_2T_2U_2+U_2\Gamma_{2+}U_2
	 &  a_1\lambda   
  U_2  vG_1v PT U_1 +U_2\Gamma_2 U_1  \\  a_1\lambda  
 U_1TP vG_1v  U_2  + U_1\Gamma_2U_2 &    U_1T_1U_1+ U_1\Gamma_1 U_1
	\end{array} \right]. 
$$
Note that $a_{22}$ is invertible. Therefore $U B(\lambda)U$ is invertible  provided the following   exists  
\begin{multline*}
d=\Bigg( -\frac{2i \|V\|_{L^1}}{5!(8\pi)^2} \lambda^2  U_2T_2U_2 - a_1^2\lambda^2 S_2  vG_1v PT U_1 (U_1T_1U_1)^{-1}  U_1TP vG_1v  U_2 + U_2 \Gamma_{2+} U_2 \Bigg)^{-1}
\\ =- \frac{5!(8\pi)^2}{2i  \|V\|_{L^1}\lambda^2}\Bigg( U_2 T_2 U_2- \frac{10}{3\|V\|_{L^1}} U_2  vG_1v PT U_1 (U_1T_1U_1)^{-1}  U_1TP vG_1v  U_2  \Bigg)^{-1}+U_2\Gamma_{-2+}U_2.
\end{multline*}
Note that, since $S_2v=S_2x_jv=0$ we can rewrite the operator in parenthesis as
\begin{multline*}
U_2T_2U_2 - \frac{10\la (U_1T_1U_1)^{-1} U_1Tv,U_1Tv \ra}{3\|V\|_{L^1}}   U_2  v W v U_2 \\
= U_2vG_3vU_2+\frac{10}{3 }  \Big(1-\frac{\la (U_1T_1U_1)^{-1} U_1Tv,U_1Tv \ra}{\|V\|_{L^1}}\Big) U_2vWvU_2.
\end{multline*}
Note that by Lemma~\ref{lem:S_3} below the kernel of $T_2$ agrees with the kernel of $S_2vG_3vS_2$. Therefore $U_2vG_3vU_2$ is invertible and positive definite. Since $U_2vWvU_2$ is positive semi-definite, the inverse exists if we can prove that $\la (U_1T_1U_1)^{-1} U_1Tv,U_1Tv \ra\leq \|V\|_{L^1}$.
Note that 
$$
U_1TPTU_1 u = \frac1{\|V\|_{L^1}} U_1(Tv) \la u,U_1(Tv)\ra.  
$$
Also note that $U_1T_1U_1-U_1TPTU_1$ is positive semi-definite. Therefore the required bound follows from the following lemma with $\mathcal H=U_1L^2$, $z= U_1(Tv)$,  $\alpha=\frac1{\|V\|_{L^1}}$,  and $\mathcal S =U_1T_1U_1-U_1TPTU_1$.  
\begin{lemma}\label{lem:posdef} Let $\mathcal H$ be a Hilbert space. Fix $z\in \mathcal H$ and $\alpha>0$ and let $\mathcal T(u)=\alpha z\la u,z\ra$, $u\in\mathcal H$. Let $\mathcal S$ be a positive semi-definite operator on $\mathcal H$  so that $\mathcal T+\mathcal S$ is invertible.   Then,
$$
0\leq \big\la (\mathcal  T+\mathcal  S)^{-1}z, z\big\ra \leq \frac1\alpha.
$$   
\end{lemma}
\begin{proof} 
Let $w=(\mathcal  T+\mathcal  S)^{-1}z$. We have
$$
z=\mathcal  Tw+\mathcal  Sw=\alpha z\la w,z\ra +\mathcal  Sw, \text{ and hence } \mathcal  Sw=z-\alpha z\la w,z\ra.
$$
 Then since $\mathcal  S$ is positive semi-definite, 
$$
0\leq \la \mathcal  Sw,w\ra = \big\la z-\alpha z\la w,z\ra,w\big\ra =\la z,w\ra-\alpha |\la z,w\ra|^2. 
$$
Therefore, $\la (\mathcal  T+\mathcal  S)^{-1}z,z\ra =\la w,z\ra\in \R$ and
$$
0\leq \la w, z\ra \leq \frac1\alpha.   \quad  \qedhere
$$ 
\end{proof}

We conclude that
$$
d=\lambda^{-2}U_2DU_2+U_2\Gamma_{-2+}U_2=U_2\Gamma_{-2}U_2.
$$
 Using this in the Feshbach formula \eqref{feshbach}, we obtain
\begin{multline} \label{eq:UBUinverse}
		 (UB(\lambda)U)^{-1} =-\frac1{g(\lambda)}\left[\begin{array}{cc} U_2\Gamma_{-2} U_2 +U_2\Gamma_{-1}U_2 & U_2\Gamma_{-1 }U_1\\
		 U_1 \Gamma_{-1} U_2 & U_1 \Gamma_{0} U_1 
		\end{array}\right] \\ = U_2\Gamma_{-3} U_2+U_2\Gamma_{-2}U_1+U_1\Gamma_{-2}U_2+U_1\Gamma_{-1}U_1. 
	\end{multline}

We now focus on the case $S_3=0$, $U_2=S_2\neq 0$. We have $B(\lambda)^{-1}=(UB(\lambda)U)^{-1}$. Using \eqref{M+S1ex2} and orthogonality relations \eqref{ort:1}-\eqref{ort:6}, we have 
$$ 
S_2(M  (\lambda)+S_1)^{-1} =(M  (\lambda)+S_1)^{-1}S_2 =  S_2    +\Gamma_2.
$$
Also recall that 
$$
(M  (\lambda)+S_1)^{-1} = QD_0Q+\Gamma_1.
$$
Using these in \eqref{Minversegeneral}, we have  
\begin{multline}\label{MinverseS2nonzero}
M(\lambda)^{-1}= 
QD_0Q+\Gamma_1\\ + (M  (\lambda)+S_1)^{-1}\big[U_2\Gamma_{-3} U_2+U_2\Gamma_{-2}U_1+U_1\Gamma_{-2}U_2+U_1\Gamma_{-1}U_1\big](M  (\lambda)+S_1)^{-1}\\= S_2\Gamma_{-3} S_2+S_2\Gamma_{-2} Q+ Q\Gamma_{-2}S_2 + S_2\Gamma_{-1}+\Gamma_{-1}S_2 +Q  \Gamma_{-1}Q +Q  \Gamma_0+  \Gamma_0Q +\Gamma_1.
\end{multline}
 This expansion is valid also in the case $U_1=0$, proving Theorem~\ref{thm:Minvexp} in the case of resonance of the second kind.

We consider the final case, when $S_3\neq 0$.
Using 
$$
(A^\pm(\lambda)+S_1)^{-1}   S_3 =S_3 (A^\pm(\lambda)+S_1)^{-1}= S_3, 
$$
$$
S_3 M_0^{\pm}=M_0^{\pm}S_3= \Gamma_3,
$$
$$
S_3 M_0^{\pm}S_3= \lambda^4 S_3 vG_4v S_3+\Gamma_5 =\lambda^4 T_3+\Gamma_5,
$$
we have
$$
S_3B^\pm  S_3=  - S_3 \sum_{k=1}^4 (-1)^k [ M_0^{\pm}(\lambda) (A^\pm(\lambda)+S_1)^{-1}]^k  S_3   +\Gamma_5\\ =\lambda^4 T_3  +\Gamma_5,
$$
 provided that $v(x)\les \la x \ra^{-\frac{15}{2}-}$.  
 

If $U\neq 0$, we invert $B(\lambda)$ using Feshbach's formula for the block form \eqref{eq:BlambdaS3}. Note that
$$
d=\big(S_3BS_3-S_3BU(UBU)^{-1}UBS_3\big)^{-1}.
$$  
The leading term is $\lambda^4 T_3  +\Gamma_5$. We write the second term as 
\begin{multline*}
S_3BU_2(UBU)^{-1}U_2BS_3+S_3BU_1(UBU)^{-1}U_2BS_3\\+S_3BU_2(UBU)^{-1}U_1BS_3+S_3BU_1(UBU)^{-1}U_1BS_3=\Gamma_5.
\end{multline*}
To obtain the estimate, we used $S_3BU_2=\Gamma_4$, $S_3BU_1=\Gamma_3$, and \eqref{eq:UBUinverse}. Therefore, for small $\lambda>0,$
$$
d=\lambda^{-4}D_3+S_3\Gamma_{-3}S_3=S_3\Gamma_{-4}S_3.
$$
Using this in Feshbach's formula  for the block form \eqref{eq:BlambdaS3} we obtain
\begin{multline*}
B(\lambda)^{-1}=\lambda^{-4}D_3+S_3\Gamma_{-3}S_3 + S_3\Gamma_{-4}S_3BU (UBU)^{-1}+(UBU)^{-1}UBS_3\Gamma_{-4}S_3 \\+\lambda^{-4}(UBU)^{-1}UBS_3\Gamma_{-4}S_3BU (UBU)^{-1}+(UBU)^{-1}.
\end{multline*} 
Using \eqref{eq:UBUinverse}, decomposing $U=U_1+U_2$ as above, and using $S_3BU_2=\Gamma_4$, $S_3BU_1=\Gamma_3$,  we have 
$$
B(\lambda)^{-1}= \lambda^{-4}D_3 + S_2\Gamma_{-3} S_2+ S_2 \Gamma_{- 2} S_1+ S_1 \Gamma_{- 2} S_2  +  S_1  \Gamma_{- 1} S_1.
$$
Finally, using
$$ 
S_2(M  (\lambda)+S_1)^{-1} =(M  (\lambda)+S_1)^{-1}S_2 =  S_2    +\Gamma_2,
$$
$$ 
S_3(M  (\lambda)+S_1)^{-1} =(M  (\lambda)+S_1)^{-1}S_3 =  S_3    +\Gamma_3,
$$ 
$$
(M  (\lambda)+S_1)^{-1} = QD_0Q+\Gamma_1,
$$
we obtain Theorem~\ref{thm:Minvexp} in the case of a resonance of the third kind.

\section{Low energy dispersive estimates}\label{sec:low disp}

In this section we analyze the perturbed evolution $e^{-itH}$ in $L^1 \rar L^{\infty}$ setting for small energy, when the spectral variable $\lambda$  is in a small neighborhood of the threshold energy $\lambda=0$. As in the free case, we represent the solution via Stone's formula, \eqref{stone}. As usual, we analyze \eqref{stone} separately for large energy, when $\lambda \gtrsim 1$, and for small energy, when $ \lambda \ll 1$, see for example \cite{Sc2, eg2}.  The effect of the presence of zero energy resonances   is only felt in the small energy regime.  Different resonances change the asymptotic behavior of the perturbed resolvents and hence that of the spectral measure as $\lambda\to 0$ which we study in this section.  The large energy argument appears in Section~\ref{sec:large} to complete the proof of Theorem~\ref{thm:main}.

We start with the following lemma which will be used repeatedly.  
\begin{lemma} \label{lem:QvR} Assume that $v(x)\les\la x\ra^{-\frac52-},$ then
$$
\sup_y \big\| [Qv R^{\pm}(H_0, \lambda^4)](\cdot,y)\big\|_{L^2} \les 1,\,\,\text{ and }\,\, 
\sup_y \big\| \partial_\lambda [Qv R^{\pm}(H_0, \lambda^4)](\cdot,y)\big\|_{L^2} \les \frac1\lambda.
$$
Assuming that $v(x)\les\la x\ra^{-\frac72-},$ we have
$$
\sup_y \big\| \big[S_2v (R^{\pm}(H_0, \lambda^4)-G_0) \big](\cdot,y)\big\|_{L^2} \les \lambda,\,\, \,\, 
\sup_y \big\| \partial_\lambda \big[S_2v (R^{\pm}(H_0, \lambda^4)-G_0) \big](\cdot,y) \big\|_{L^2} \les 1,
$$
and  
$$
\sup_y \big\| \big[S_2v (R^{+}(H_0, \lambda^4) -R^{-}(H_0, \lambda^4) ) \big](\cdot,y)\big\|_{L^2} \les \lambda, $$
$$
\sup_y \big\| \partial_\lambda \big[S_2v  (R^{+}(H_0, \lambda^4) -R^{-}(H_0, \lambda^4) )\big](\cdot,y) \big\|_{L^2} \les 1.
$$
\end{lemma}
\begin{proof} We prove the assertion for $+$ sign. Recall the expansion \eqref{eq:R0low}. Using the fact $Qv=0$ we have 
\begin{multline*}
[Q vR^{+}(H_0, \lambda^4)] (y_2, y) = \frac{1}{8 \pi \lambda} \int_{\R^3} Q(y_2, y_1) v(y_1) [F(\lambda |y-y_1|) -F(\lambda|y|)]dy_1\\
=\frac{1}{8 \pi } \int_{\R^3} Q(y_2, y_1) v(y_1) \int_{|y|}^{|y-y_1|}F^\prime(\lambda s) ds dy_1,
\end{multline*}
 where 
 $$ F (p) = \frac{ e^{ip} - e^{-p}}{ p}.  $$ 
 Noting that $|F^\prime(p)|\les 1$, and using the absolute boundedness of $Q$, we obtain
 $$
\big\| [Q vR^{+}(H_0, \lambda^4)] (\cdot, y) \big\|_{L^2}\les \Big\| \int_{\R^3} |Q(y_2, y_1)| |v(y_1)|\la y_1\ra   dy_1 \Big \|_{L^2_{y_2}}\les \| v(y_1) \la y_1\ra \|_{L^2}\les 1,
 $$
 uniformly in $y$. 
 
 Now consider $S_2v (R^{+}(H_0, \lambda^4)-G_0)$. We have
 \begin{align*}
[S_2 v(R^{+}(H_0, \lambda^4)-G_0)] (y_2, y)= \frac{1}{8 \pi \lambda} \int_{\R^3} S_2(y_2, y_1) v(y_1) F(\lambda |y-y_1|) dy_1
\end{align*}
 where 
 $$ F (p) = \frac{ e^{ ip} - e^{-p}}{ p}  + p.$$ 
 Noting  that $S_2v =0$ we can rewrite the integral above as 
\begin{multline*}
\frac{1}{8 \pi \lambda} \int_{\R^3} S_2 (y_2,y_1) v(y_1) [F(\lambda |y-y_1|) - F(\lambda |y| ) ]dy_1\\
=\frac{1}{8 \pi }   \int_{\R^3} S_2(y_2,y_1) v(y_1) \int_{|y|}^{| y-y_1|} F^{\prime} (\lambda s ) ds dy_1.
\end{multline*}

Furthermore, one has $ S_2y_jv =0$,   and hence 
$$  \int_{\R^3}  S_2 (y_2,y_1) v(y_1) \frac{y_1 \cdot y}{|y|}  F^{\prime} (\lambda |y| ) dy_1= \int_{\R^3}  S_2 (y_2,y_1) v(y_1)   y_1 dy_1  \cdot \frac{ y }{|y|} F^{\prime} (\lambda |y| )  =0. $$

This gives 
\begin{multline} 
   \int_{\R^3}  S_2 (y_2,y_1)  v(y_1)  \int_{|y|}^{| y-y_1|}  F^{\prime} (\lambda s ) ds dy_1   \\
  = \int_{\R^3}  S_2 (y_2,y_1) v(y_1)  \Big[\int_{|y|}^{| y-y_1|} F^{\prime} (\lambda s ) ds  + \frac{ y_1 \cdot y }{|y|} F^{\prime} (\lambda |y| )   \Big] dy_1    \\
 = \int_{\R^3}  S_2 (y_2,y_1) v(y_1)   \Bigg[ \int_{|y|- \frac{ y_1 \cdot y}{|y|}}^{| y-y_1|} F^{\prime} (\lambda s ) ds - \int_{|y|- \frac{ y_1 \cdot y}{|y|}}^{|y|} F^{\prime} (\lambda s ) ds+ \int_{|y| - \frac{ y_1 \cdot y}{|y|}}^{ |y| } F^{\prime} (\lambda |y| ) ds \Bigg] dy_1   \\ \label{seperates}
 = \int_{\R^3}  S_2 (y_2,y_1) v(y_1)   \Bigg[\int_{|y|- \frac{ y_1 \cdot y}{|y|}}^{| y-y_1|} F^{\prime} (\lambda s ) ds +  \lambda \int_{|y| - \frac{y_1 \cdot y}{|y|}}^{|y| } \int_{s}^{|y|} F^{\prime \prime} ( \lambda k) dk ds \Bigg] dy_1. 
\end{multline} 
To control the integrals in \eqref{seperates} notice that $| F^{(k)} (p) | \les p^{2-k} $ for $k= 1,2$. Therefore, for $|y| - \Big| \frac{y_1 \cdot y}{|y|} \Big| \geq 0$, we obtain
\begin{align} \label{Ffirstests}
\Big| \int_{|y| -\frac{y_1 \cdot y}{|y|}}^{|y-y_1|} F^{\prime} (\lambda s ) ds \Big| \les \lambda \Big|\int_{|y| -\frac{y_1 \cdot y}{|y|}}^{|y-y_1|} s ds\Big|   \les \lambda \Big| |y-y_1|^2 - (|y| - \frac{y_1 \cdot y}{|y|})^2 \Big| \les \lambda \la y_1 \ra^{2} .
\end{align}
Note that if $|y| - \Big| \frac{y_1 \cdot y}{|y|} \Big| < 0$, one has  $ |y| ,|y-y_1| < |y_1|$ and therefore the above inequality is trivial. 

For the second term in \eqref{seperates}, we have
\begin{align} \label{Fsecondests}
\int_{|y|-\frac{y_1 \cdot y}{|y|}}^{|y| } \int_{s}^{|y|} F^{\prime \prime} ( \lambda k) dk ds = \int_{|y|-\frac{y_1 \cdot y}{|y|}}^{|y| } [ k- |y| + \frac{y_1 \cdot y}{|y|} ] F^{\prime \prime} ( \lambda k) dk. 
\end{align}
Noting that $|[ k- |y| + \frac{y_1 \cdot y}{|y|} ]| \les \la y_1 \ra$ and $| F^{\prime \prime}(\lambda k) | \les 1$. This term can be controlled by $ \la y_1 \ra^2$. 
Finally, by  \eqref{Ffirstests} and \eqref{Fsecondests}, we obtain
 \begin{multline*}
	\big\| [S_2v (R^{+}(H_0, \lambda^4)-G_0)] (\cdot, y) \big\|_{L^2}\\
	\les \lambda \Big\| \int_{\R^3} |S_2(y_2, y_1)| |v(y_1)|\la y_1\ra^2   dy_1 \Big \|_{L^2_{y_2}}
	\les \lambda \| v(y_1) \la y_1\ra^2 \|_{L^2}\les \lambda, 
 \end{multline*}
 uniformly in $y$. 
 
To establish the bound on the first derivative, note that 
$$
\partial_{\lambda} F(\lambda r) = \frac{1}{\lambda} \Big[\frac{ [i(\lambda r)-1] e^{i(\lambda r)} + e^{-(\lambda r)} [(\lambda r)+1]}{(\lambda r)} - (\lambda r) \Big] =:\frac{1}{\lambda} \tilde{F}(\lambda r) 
$$
Since one has $| \tilde{F}^{k} (p) | \les p^{2-k}$, one can apply the  same method to $\tilde{F}$ to finish the proof.
 
 The last assertion follows from noting that the bounds used on $S_2v (R^{\pm}(H_0, \lambda^4)-G_0)$ also apply to $S_2v (R^{+}(H_0, \lambda^4)-R^{-}(H_0, \lambda^4))$, see \eqref{eq:R0low} and the subsequent discussion.
\end{proof}

We first consider the case when zero is regular ($S_1=0$) or when there is a resonance of the first kind $S_1\neq 0, S_2=0$. 
\begin{theorem}\label{thm:firstkind} Assume that $v(x)\les \la x\ra^{-\frac52-}$ and $S_1=0$, or that   $v(x)\les \la x\ra^{-\frac{7}2-}$ and  $S_1\neq 0, S_2=0$. Then 
\begin{align}\label{eq:first low int}
	\sup_{x,y\in \mathbb R^3}\Bigg|\int_0^\infty e^{it\lambda^4} \lambda^3 \chi(\lambda) R_V^\pm (\lambda^4)(x,y) \, d\lambda \Bigg|\les \la t\ra^{-\frac34}.
\end{align}
\end{theorem}
\begin{proof}Recall \eqref{resid}:
$$R^{\pm}_V(\lambda^4)= R^{\pm}(H_0, \lambda^4) - R^{\pm}(H_0, \lambda^4)v (M^{\pm} (\lambda))^{-1} vR^{\pm}(H_0, \lambda^4).
$$ 
We already obtained the required bound for the free term in Lemma~\ref{lem:free bound}. For the correction term, dropping the $\pm$ signs, the claim will follow from Lemma~\ref{lem:t-34bound} with 
\be\label{eq:mEregular} 
\mE(\lambda)(x,y)=\big[ R (H_0, \lambda^4) v  (M (\lambda))^{-1}  v R (H_0, \lambda^4)\big](x,y).
\ee 
By Theorem~\ref{thm:Minvexp}, in the regular case  we have
$ M (\lambda)^{-1}= QD_0Q+\Gamma_1.$
In the case of a resonance of the first kind,   we have
$$
M (\lambda)^{-1}=    Q\Gamma_{-1}Q+ Q\Gamma_0+\Gamma_0Q+\Gamma_1.
$$
First consider the contribution of $\Gamma_1$ to \eqref{eq:mEregular}:
$$
\big[ R (H_0, \lambda^4) v \Gamma_1 v R (H_0, \lambda^4)\big](x,y).
$$
Note that, by \eqref{eq:freeO1}  we have 
\be\label{eq:freeO1L2} \|  v R (H_0, \lambda^4)(\cdot, y)\|_{L^2}\les \frac1\lambda,\,\,\, \| \partial_\lambda v R (H_0, \lambda^4)(\cdot, y)\|_{L^2}\les \frac1{\lambda^2}
\ee 
uniformly in $y$. Therefore we estimate the contribution of the error term to $\mE(\lambda)(x,y)$ by 
$$
\lambda \|  v R (H_0, \lambda^4)(\cdot, x)\|_{L^2}\|  v R (H_0, \lambda^4)(\cdot, y)\|_{L^2}\les \frac1\lambda,
$$
and its $\lambda$ derivative by $\frac1{\lambda^2}$.  Hence, the claim follows from Lemma~\ref{lem:t-34bound} with $\alpha=1$. 

Now, consider the contribution of $ Q\Gamma_{-1}Q$ to \eqref{eq:mEregular}:
$$
 \big[ R (H_0, \lambda^4) v Q\Gamma_{-1} Q v R (H_0, \lambda^4)\big](x,y).   
$$ 
Note that, by Lemma~\ref{lem:QvR}, we bound this term by 
$$
  \|Q v R (H_0, \lambda^4)(\cdot, y)\|_{L^2} \|Q v R (H_0, \lambda^4)(\cdot, x)\|_{L^2} \| |\Gamma_{-1}| \|_{L^2\to L^2} \les \frac1\lambda 
$$ 
uniformly in $x,y$. 
Similarly, its $\lambda$-derivative  is bounded by $\frac1{\lambda^2}$. Therefore, the claim follows from Lemma~\ref{lem:t-34bound}. 

The contributions of $Q\Gamma_0$ and $\Gamma_0Q$ can be bounded similarly by using Lemma~\ref{lem:QvR} on one side and   \eqref{eq:freeO1L2}  on the other side.
\end{proof}

\begin{theorem}\label{thm:second/thirdkind} Assume that $v(x)\les \la x\ra^{-\frac{11}2-}$.
If $S_2 \neq 0$, $S_3=0$ then
\begin{align} \label{eq:secondkind}
	\sup_{x,y\in \mathbb R^3}\Bigg|\int_0^\infty e^{it\lambda^4} \lambda^3 \chi(\lambda) R_V^\pm (\lambda^4)(x,y) \, d\lambda   -F^{\pm}(x,y)\Bigg|\les \la t\ra^{-\frac34}.
\end{align}
Here   $F^\pm$ are time dependent finite rank operators satisfying  $\|F^{\pm}\|_{L^1\to L^\infty}\les \la t\ra^{-\frac14}$. 

Moreover if $v(x)\les \la x\ra^{-\frac{15}2-}$ and $S_3\neq 0$, then
\begin{align}\label{eq:thirdkind}
	\sup_{x,y\in \mathbb R^3}\Bigg|\int_0^\infty e^{it\lambda^4} \lambda^3 \chi(\lambda) [R_V^{+} (\lambda^4) - R_V^{-} (x,y)] \, d\lambda -G(x,y)\Bigg|\les \la t\ra^{-\frac12},
\end{align} 
where   $G$ is a  time dependent finite rank operator  satisfying  $\|G  \|_{L^1\to L^\infty}\les \la t\ra^{-\frac14}$.
 \end{theorem}

\begin{proof} 

We first prove \eqref{eq:secondkind}. By Theorem~\ref{thm:Minvexp}, in the case of a resonance of the second kind,  we have
$$[M^\pm(\lambda) ]^{-1} =  S_2\Gamma_{-3} S_2+S_2\Gamma_{-2} Q+ Q\Gamma_{-2}S_2 + S_2\Gamma_{-1}+\Gamma_{-1}S_2 +Q  \Gamma_{-1}Q +Q  \Gamma_0+  \Gamma_0Q +\Gamma_1. $$
We only consider the contribution of $S_2\Gamma_{-3} S_2$ to \eqref{resid}, the others can be handled similarly.
Let 
$$
\mathcal E(\lambda,x,y)=\big[R ^{\pm}(H_0, \lambda^4) v S_2\Gamma_{-3} S_2 v R^{\pm} (H_0, \lambda^4)\big](x,y)
$$   
Note that by Lemma~\ref{lem:QvR} we have 
\begin{multline*}
\mathcal E =G_0v S_2\Gamma_{-3} S_2 vG_0+G_0v S_2\Gamma_{-3} S_2 v(R^{\pm} (H_0, \lambda^4)-G_0)\\+(R^{\pm} (H_0, \lambda^4)-G_0)v S_2\Gamma_{-3} S_2 vG_0+O_1(\lambda^{-1}).
\end{multline*}
By Lemma~\ref{lem:t-34bound}, the contribution of the last term is $\les \la t\ra^{-\frac34 }$. 
Moreover, noting that $S_2v=0$, we have 
\begin{align*}
\big\|[S_2v G_0] (\cdot,y)\big\|_{L^2} =\Big\| \int_{\R^3} S_2(\cdot,y_1) v(y_1) [|y-y_1| - |y|] dy_1\Big\|_{L^2}\les 1,
\end{align*} 
since  $|[|y-y_1| - |y|] | \les \la y_1 \ra$.  Therefore, the first term is $O_1(\lambda^{-3})$, and by Lemma~\ref{lem:t-34bound} its contribution is $\les \la t\ra^{-\frac14}$. Also note that its contribution is finite rank since $S_2$ is. Similarly the contributions of second and third terms are $\les \la t\ra^{-\frac12}$, and finite rank.  One can explicitly construct the operators $F^\pm(x,y)$ from the contribution of these operators to the Stone formula, \eqref{stone}.

Next we prove \eqref{eq:thirdkind}.  
Note that all the term in $M(\lambda) ^{-1}$ in Theorem~\ref{thm:Minvexp} except  $\lambda^{-4}D_3$ are similar to the terms in the $M^{-1}(\lambda)$ that we considered in the case of resonance of the second kind. Therefore, we only control the terms interacting with $D_3$, that is we need to control the contribution of the following term to the Stone's formula, 
\begin{multline*}
 [R^{+}(H_0, \lambda^4) - R^{-}(H_0, \lambda^4)] v \frac{D_3}{\lambda^4} v R^{+}(H_0, \lambda^4)\\
= [R^{+}(H_0, \lambda^4) - R^{-}(H_0, \lambda^4)] v \frac{D_3}{\lambda^4}v G_0 
 \\+ [R^{+}(H_0, \lambda^4) - R^{-}(H_0, \lambda^4)] v \frac{D_3}{\lambda^4} v [R^{+}(H_0, \lambda^4) -G_0].
\end{multline*} 
Using Lemma~\ref{lem:QvR}, the first term is $O_1(\lambda^{-3})$, and hence  its contribution  to Stone's formula is $\la t\ra^{-\frac14}$ by Lemma~\ref{lem:t-34bound}, and is finite rank. 
Similarly, the second term is $O_1(\lambda^{-2})$ and its contribution is $ \les \la t\ra^{-\frac12}$.  $G(x,y)$ is obtained explicitly by inserting these operators in \eqref{stone}.
\end{proof}

We note that the time decay of the non-finite rank portion of the evolution when $S_3\neq0$ can be improved at the cost of spatial weights.  

\begin{corollary}\label{cor:wtd decay}
	
	If $v(x)\les \la x\ra^{-\frac{15}2-}$ and $S_3\neq 0$, then
	\begin{align}\label{eq:thirdkind2}
	\Bigg|\int_0^\infty e^{it\lambda^4} \lambda^3 \chi(\lambda) [R_V^{+} (\lambda^4) - R_V^{-} (x,y)] \, d\lambda -G(x,y)\Bigg|\les \la t\ra^{-\frac34}\la x\ra^{\f52}\la y \ra^{\f52},
	\end{align} 
	where   $G$ is a  time dependent finite rank operator  satisfying  $\|G  \|_{L^1\to L^\infty}\les \la t\ra^{-\frac14}$.
	
\end{corollary}

\begin{proof}
	
	We need only supply a new bound for the contribution of the following
	\begin{align}
		[R^{+}(H_0, \lambda^4) - R^{-}(H_0, \lambda^4)] v \frac{D_3}{\lambda^4} v [R^{+}(H_0, \lambda^4) -G_0].
	\end{align}
	We note that $S_3v P_2(x)=0$ for any quadratic polynomial in the $x_j$ variables.  Hence, $S_3vG_1=0$ as we may write $G_1(x,y)=|x|^2-2x\cdot y+|y|^2$.  By truncating the expansion  in \eqref{eq:R0low} earlier, we see
	$$
		[R^+(H_0,\lambda^4)-R^-(H_0,\lambda^4)]=\frac{a^+-a^-}{\lambda}+(a_1^+-a_1^-)\lambda G_1 +O((\lambda|x-y|)^\ell |x-y|) \qquad 1<\ell\leq 3.
	$$
	Using the orthogonality relations above and selecting $\ell=\f32$,   one can see that
	$$
		[R^+(H_0,\lambda^4)-R^-(H_0,\lambda^4)](x, \cdot) vS_3= O_1(\lambda^{\f32}\la x \ra^{\f52})
	$$
	A very similar computation shows that
	$$
		S_3v [R^{+}(H_0, \lambda^4) -G_0](\cdot, y)= O_1(\lambda^{\f32}\la y\ra^{\f52}).
	$$
	Combining these, we see that
	$$
		[R^{+}(H_0, \lambda^4) - R^{-}(H_0, \lambda^4)] v \frac{D_3}{\lambda^4} v [R^{+}(H_0, \lambda^4) -G_0]= O_1(\lambda^{-1} \la x\ra^{\f52} \la y \ra^{\f52}).
	$$
	Applying Lemma~\ref{lem:t-34bound} proves the claim.
	
\end{proof}

\section{The Perturbed Evolution For Large Energy } \label{sec:large}

For completeness, we include a proof of the dispersive bound for the large energy portion of the evolution.  Here we need to assume the lack of eigenvalues embedded in $[0,\infty)$ for the perturbed fourth order operator $H=(-\Delta)^2+V$.  It is known that embedded eigenvalues may exist even for compactly supported smooth potentials.   To complete the proof of Theorem~\ref{thm:main} we show

\begin{prop} \label{prop:large}Let $|V(x) | \les \la x \ra^{-3-}$, and assume there are no embedded eigenvalues in the continuous spectrum of $H$, then 
\begin{align} \label{eq:large}
\sup_{x,y\in \mathbb R^3}\Bigg|\int_0^\infty e^{it\lambda^4} \lambda^3 \widetilde{\chi}(\lambda) R_V^{\pm} (\lambda^4)  (x,y) \, d\lambda \Bigg|\les | t |^{-\frac34}.
\end{align}
\end{prop} 

To prove the Proposition~\ref{prop:large} we use the resolvent identities and write, 
\begin{align}\label{large symmetric}
R_V(\lambda^4)= R^\pm(H_0, \lambda^4)  - R^\pm(H_0, \lambda^4) VR^\pm (H_0, \lambda^4) + R^\pm(H_0, \lambda^4) VR_V(\lambda^4) V R^\pm (H_0, \lambda^4).
\end{align}

Recall by the second part of Remark~\ref{rmk:large}, we know that the first summand in \eqref{large symmetric} satisfies the bound in \eqref{eq:large}. Therefore, it suffices to establish the bound in Proposition~\ref{prop:large} is valid for the last two summands in \eqref{large symmetric}. 
Recall by \eqref{eq:freeO1}, we have 
\begin{align} \label{est1}
R^\pm(H_0, \lambda^4)  (x,y)  =O_1(\lambda^{-1}). 
\end{align}
This, along with the fact that $\lambda\gtrsim 1$, shows that 
$$
	R^\pm (H_0, \lambda^4)VR^\pm(H_0, \lambda^4)=O_1(\lambda^{-1}),
$$
as the following bounds hold uniformly in  $x,y$:
 \begin{align*}
&|R^\pm (H_0, \lambda^4)VR^\pm(H_0, \lambda^4)(x,y)| \les \lambda^{-1} \int_{\R^3} |V (x_1)| dx_1\les \lambda^{-1} \\
&| \partial_\lambda\{ R^\pm (H_0, \lambda^4)VR^\pm(H_0, \lambda^4)]\}(x,y)| \les \lambda^{-2} \int_{\R^3} |V (x_1)| dx_1 \les \lambda^{-2}.
 \end{align*}
Hence, by first part of Remark~\ref{rmk:large}, $R^\pm VR^\pm$ contributes $|t|^{-\f34}$ to Stone's formula. 

We next consider the last  term in \eqref{large symmetric}. To control this term, we utilize the following.
\begin{theorem} \label{th:LAP}\cite[Theorem~2.23]{fsy} Let $|V(x)|\les \la x \ra ^{-k-1}$. Then for any $\sigma>k+1/2$, $\partial_z^k R_V(z) \in \mathcal{B}(L^{2,\sigma}(\R^d), L^{2,-\sigma}(\R^d))$ is continuous for $z \notin {0}\cup \Sigma$. Further, 
\begin{align*} 
\|\partial_z^k  R_V(z) \|_{L^{2,\sigma}(\R^d) \rar L^{2,-\sigma}(\R^d)} \les z^{-(3+3k)/4}. 
\end{align*}
\end{theorem}	

The following suffices to finish the proof of Proposition~\ref{prop:large}. 
\begin{lemma} Let $|V(x) | \les \la x \ra^{-3-}$, then 
\begin{align*}
\sup_{x,y\in \mathbb R^3}
\Big| \int_{0}^{\infty} e^{-it\lambda^4} \widetilde\chi(\lambda) \lambda^{3} [R^\pm(H_0, \lambda^4) V R_V^\pm(\lambda^4) VR^\pm(H_0, \lambda^4)](x,y) d \lambda \Big| \les |t|^{-\f34},
\end{align*}
\end{lemma} 
\begin{proof} Recalling the proof of Lemma~\ref{lem:t-34bound}, it suffices to establish
\begin{align*}
 &\|R^\pm(H_0, \lambda^4) VR_V(\lambda^4) V R^\pm(H_0, \lambda^4) \|_{L^1 \rightarrow L^{\infty}} \les \lambda^{-1} \\
& \| \partial_{\lambda} \{ R^\pm(H_0, \lambda^4) VR_V(\lambda^4) V R^\pm(H_0, \lambda^4) \} \|_{L^1 \rightarrow L^{\infty}} \les \lambda^{-2}
\end{align*}

Note that first by \eqref{est1}, and using that $L^\infty \subset L^{2,-\f32-}$, we have 
\begin{align}
\| [R^\pm (H_0, \lambda^4)] \|_{ L^1 \rightarrow L^{2,-\sigma} }= O_1(\lambda^{-1}), \,\,\,\ \sigma>3/2,
\end{align}
along with the dual estimate as an operator from $L^{2,\sigma}\to L^\infty$.
Hence, by Theorem~\ref{th:LAP} we have the following estimate
\begin{align*}
& \| [R^\pm (H_0, \lambda^4) VR_V(\lambda^4) V R^\pm (H_0, \lambda^4)] \|_{L^1 \rightarrow L^{\infty}} \\ 
&\les \|R^\pm (H_0, \lambda^4) \|_{ L^{2,\sigma} \rightarrow L^{\infty}} \| V R_V(\lambda^4) V \|_{ L^{2,- \sigma} \rightarrow L^{2,\sigma} } \| R^\pm (H_0, \lambda^4) \|_{ L^ {1} \rightarrow L^{2,-\sigma} } \les \lambda^{-1}
\end{align*}
for any $|V(x) | \les \la x \ra^{-2-}$.  In fact, one can show this term is smaller, though this bound is valid since $\lambda \gtrsim 1$.
Similarly, by \eqref{est1} and Theorem~\ref{th:LAP} with $z=\lambda^4$ one obtains
\begin{align*}
 \| \partial_{\lambda} \{R^\pm(H_0, \lambda^4) VR_V(\lambda^4) V R^\pm(H_0, \lambda^4) \}\|_{L^1 \rightarrow L^{\infty}} \les \lambda^{-2}
\end{align*}
for any $|V(x) | \les \la x \ra^{-3-}$. Here, we note that the extra decay on $V$ is needed when the derivative falls on the perturbed resolvent $R_V$ so that $V$ maps $L^{2,-\f32-}\to L^{2,\f32+}$.

\end{proof}

\section{Classification of threshold spectral subspaces}\label{sec:classification}
In this section  we establish the relationship between the
spectral subspaces $S_i L^2(\R^3)$ for $i=1,2,3$ and distributional solutions to $H\psi =0$.

\begin{lemma}\label{lem:esa1}
Assume $|v(x)| \les \la x \ra ^{-\frac{5}{2}-}$,  if $\phi \in S_1 L^2(\R^3) \setminus \{0\} $, then $\phi= Uv \psi $  where $\psi \in L^{\infty}$, $H\psi=0$ in distributional sense, and 
\be \label{eq:psi def}
	\psi= c_0 - G_0v \phi,\,\, \text{where} \,\,\ c_0= \frac{1}{ \| V\|_{L^1}} \la v,T \phi \ra .
\ee 
\end{lemma}  
\begin{proof} Assume $\phi \in S_1 L^2(\R^4)$, one has $QTQ\phi=Q(U+vG_0v)\phi=0$. Note that 
\begin{align*}
	0=Q(U+vG_0v)\phi & = (I-P)(U+vG_0 v)\phi \\
	& = U\phi + vG_0v\phi - PT\phi  \\ 
    &\implies  \phi = Uv ( -G_0v\phi + c_0)=Uv\psi\,\, \text{where} \,\, c_0= \frac{1}{ \| V\|_{L^1}} \la v,T \phi \ra.
\end{align*}
To show $ [\Delta^2+ V] \psi = [\Delta^2+ V]( -G_0v\phi + c_0)=0 $, notice that by differentiation under the integral sign 
$$\Delta^2 G_0v\phi=-\Delta\int \frac{1}{4\pi |x-y|} v(y)\phi(y)dy.$$
Since $ \frac{1}{4\pi |x-y|}$ is the Green's function for $-\Delta$, we have $ \Delta^2G_0 v\phi=v\phi$ in the sense of distributions.    
  Hence,
$$ [\Delta^2+ V] ( -G_0v\phi + c_0) = -v\phi + vUv\psi =0.
$$
Next, we show that $G_0v \phi \in L^{\infty}$. Noting that $S_1\leq Q$, we have $P \phi  =0$ and hence
\begin{align}\label{psibounded}
\Big|\int_{\R^3}  |x-y| v(y) \phi(y) dy \Big| = \Big|\int_{\R^3} [ |x-y| -|x| ] v(y) \phi(y) dy \Big|\les \int_{\R^3} \la y \ra |v(y) \phi(y)| dy < \infty.
\end{align} 
\end{proof}

The following lemma gives  further information for the function $\psi$ in Lemma~\ref{lem:esa1}.

\begin{lemma} \label{lem:psiexp2} Let $|v(x)| \les \la x \ra^{-\frac{11}4- }$. Let $\phi= Uv \psi \in S_1 L^2(\R^3) \setminus \{0\}$ as in Lemma~\ref{lem:esa1}. Then,
\be \label{eq:psi def2}
	\psi= c_0 +\sum_{i=1}^3c_i\frac{x_i}{\la x\ra}+ \sum_{1\leq i\leq j\leq 3} c_{ij} \frac{x_ix_j}{\la x\ra^3} 
	+ \widetilde\psi,  
\ee
where $c_0= \frac{1}{ \| V\|_{L^1}} \la v,T \phi \ra$ and $\widetilde\psi\in L^2\cap L^\infty$. 
Moreover, $\psi \in L^p$ for $3<p \leq \infty$ if and only if $PT\phi=0$ and $\int y  v(y)\phi(y)dy =0$. 

Furthermore, 
$\psi \in L^p$ for $2\leq p \leq \infty$ if and only if $PT\phi=0$,  $\int y  v(y)\phi(y)dy =0$, and $\int y_i y_j v(y)\phi(y)dy =0$, $1\leq i\leq j\leq 3$. 
\end{lemma}

\begin{proof} Note that all the terms in the expansion and the function $\psi$ are in $L^\infty$, therefore it suffices to prove the claim for $|x|>1$.
Using Lemma~\ref{lem:esa1} and the fact  that $P\phi=0$, we write  
\begin{multline*} 
\psi(x)-c_0 = -\frac{1}{8\pi}\int_{\R^3}  |x-y|  [v\phi](y) dy  \\
=-\frac{1}{8\pi}\int_{\R^3} \Big( |x-y| - |x| + \frac{ x \cdot y} { |x| } + \frac{ |y|^2} { 2|x|} - \frac {(x \cdot y )^2}{|x|^3} \Big) [v\phi](y) dy\\
+\frac{1}{8\pi}\int_{\R^3} \Big(\frac{ x \cdot y} { |x| } + \frac{ |y|^2} { 2|x|} - \frac {(x \cdot y )^2}{|x|^3} \Big) [v\phi](y) dy=:\psi_1+\psi_2. 
\end{multline*}

We first claim that $\psi_1\chi_{B(0,1)^c}\in L^2\cap L^\infty$.
To prove this claim we first consider the case $ |y| <  |x|/2 $. In this case, by a Taylor  expansion  we have 
\begin{multline}
|x-y|= |x| \Big( 1 - \frac{ x \cdot y} {|x|^2}  + \frac{ |y|^2} { 2|x|^2} - \frac{1}{8} \Big(-\frac{ x \cdot y} { |x|^2}  + \frac{ |y|^2} { 2|x|^2} \Big)^2 \Big) +   O ( |y|^{3}/|x|^{2}) \\
= |x| - \frac{ x \cdot y} {|x|} + \frac{ |y|^2} { 2|x|} - \frac {(x \cdot y )^2}{8|x|^3} +  O \Big( \frac{|y|^{{\f 52}+}}{|x|^{{\f 32}+}} \Big)
\end{multline} 
Using this and the fact that $|v\phi|=|V\psi|\les v^2$ we have 
\begin{multline*}
	\Big|\int_{|y| < |x|/2} \Big( |x-y| - |x| + \frac{ x \cdot y} { |x| } + \frac{ |y|^2} { 2|x|} - \frac {(x \cdot y )^2}{8|x|^3} \Big) [v\phi](y) dy \Big| \\ 
	\les  \int_{|y| <  |x|/2}  \frac{|y|^{{\f 5 2}+}}{ |x|^{{\f 3 2}+}}  \la y\ra^{-\frac{11}2-} dy
	\les |x|^{-\f32-} \int_{\R^3} \la y\ra^{-3-}\, dy
	 \les |x|^{-{\f 32}-} 
\end{multline*}
which belongs to $L^2\cap L^\infty$ on $B(0,1)^c$. 

In the case  $ |y| >  |x|/2$, we have 
\begin{multline*}
\Big|\int_{|y| > |x|/2} \Big( |x-y| - |x| + \frac{ x \cdot y} { |x| } + \frac{ |y|^2} { 2|x|} - \frac {(x \cdot y )^2}{8|x|^3} \Big) [v\phi](y) dy\Big|\\
\les \int_{|y| > |x|/2} \Big( |y|+\frac{|y|^2}{|x|}\Big) |y|^{-\frac{11}2-} dy
\les |x|^{-\frac32-}, 
\end{multline*}
which yields the claim. 

Now note that for $|x|>1$
\begin{multline}\label{psi2exp}
8\pi \psi_2=  \sum_{i=1}^3 \frac{ x_i} { |x| } \int_{\R^3} y_i[v\phi](y) dy + \frac{ 1} { 2|x|}  \int_{\R^3} |y|^2 [v\phi](y) dy -\sum_{i,j=1}^3 \frac{x_ix_j}{|x|^3}\int_{\R^3} y_i y_j  [v\phi](y) dy. 
\end{multline}
This yields the expansion for $\psi$ since for $|x|>1$, $\frac{x_i}{|x|}-\frac{x_i}{\la x\ra}=O(|x|^{-2})$ and 
$\frac{x_ix_j}{|x|^3}-\frac{x_ix_j}{\la x\ra^3}=O(|x|^{-2} ).$

Noting that the second and third terms in \eqref{psi2exp} are in $L^p_{B(0,1)^c}$ for $3<p\leq \infty$, we see that $\psi\in L^p$, $3<p\leq \infty$, if and only if  
$$c_0+\frac1{8\pi}\sum_{i=1}^3 \frac{ x_i} { |x| } \int_{\R^3} y_i[v\phi](y) dy \in L^{3+}_{B(0,1)^c},$$
which is equivalent to $c_0=0$ and $\int_{\R^3} y_i[v\phi](y) dy=0,$ $i=1,2,3$. To obtain the final claim, to determine if $\psi\in L^2_{B(0,1)^c}$ we rewrite the last two terms in \eqref{psi2exp} as follows 
$$
\frac{1}{|x|^3}\Big(  \sum_{i=1}^3\big(\tfrac{|x|^2}2-x_i^2\big) \int_{\R^3} y_i^2 [v\phi](y) dy  -2\sum_{1\leq i<j\leq 3}  x_ix_j  \int_{\R^3} y_i y_j  [v\phi](y) dy\Big).
$$
Note that the term in the parentheses is a degree 2 polynomial in  $x$, and hence cannot be in $L^2$ unless all coefficients are zero, which implies the final claim. 
\end{proof}

The following lemma is the converse of Lemma~\ref{lem:esa1}.

\begin{lemma} Let $|v(x)| \les \la x \ra ^{-\frac{11}4-}$.   Assume that a nonzero function $ \psi\in L^\infty$   solves $H\psi=0$ in the sense of distributions. Then $\phi= Uv\psi \in S_1L^2$, and we have $\psi = c_0 - G_0 v \phi$, 
$c_0=\frac{1}{ \| V\|_{L^1}} \la v,T \phi \ra$. In particular, the expansion given in Lemma~\ref{lem:psiexp2} is valid.  
\end{lemma}
\begin{proof}
Let $ \psi\in L^\infty$   be a solution of  $H\psi=0$, or equivalently $- \Delta^2 \psi = V \psi$. We first show that for $\phi= Uv\psi\in QL^2$, namely  
$$
\int_{\R^3} v(x) \phi(x) dx =0.
$$
Note that $v\phi=V\psi \in L^1$.  Let $\eta(x)$ be a smooth cutoff function with $\eta(x)=1$ for all $|x| \leq 1$. For  $\delta>0$, let $\eta_\delta(x)=\eta(\delta x)$. We have 
$$
	|\la v \phi  , \eta_\delta \ra|= |\la  V\psi,  \eta_\delta  \ra| =| \la   \Delta^2 \psi ,  \eta_\delta  \ra| =| \la    \psi ,  \Delta^2 \eta_\delta  \ra|   \leq  \|\psi\|_{L^\infty} \| \Delta^2 \eta_\delta\|_{L^1} \les \delta.  
$$
Therefore, taking $\delta\to 0$ and using   the dominated convergence theorem we conclude that $\la v ,\phi \ra=0$.

Moreover, let $ \tilde{\psi} = \psi+ G_0v\phi$, then by assumption and \eqref{psibounded}, $ \tilde{\psi}$ is bounded and $\Delta^2 \tilde{\psi}=0$. By Liouville's theorem for biharmonic functions on $\R^n$, $  \tilde{\psi} =c $.   This implies that   $\psi= c- G_0v\phi$. Since 
$$
0=H \psi=[ \Delta^2 +V] \psi = Vc-Uv(U+v G_0v)\phi \Rightarrow v^2c= vT\phi,  
$$
we have  $c=c_0= \frac{1}{ \| V\|_{1}} \la v,T \phi \ra$.
 Lastly notice that, 
\begin{multline*}
	Q(U+vG_0v) Q \phi = Q(U+vG_0v) \phi = Q (U \phi + vG_0v \phi )\\
	= Q (U \phi -v \psi +c_0v)
	=Q(c_0v) = 0, 
\end{multline*}
hence $\phi \in S_1L^2$ as claimed. 
\end{proof}

Let $T_1= S_1 TPTS_1 - \frac{\|V\|_1}{3 (8\pi)^2} S_1 v G_1 v S_1 $, and $S_2$ be the Riesz projection on the the kernel of $T_1$.  Moreover, let $S^\prime_2$ be the Riesz projection on the the kernel of $S_1 TPTS_1$ and $S^{\prime \prime} _2$ be the Riesz projection on the the kernel of $ S_1 v G_1 v S_1$. 

\begin{lemma}   \label{lem:S_2} Let $|v(x)| \les \la x \ra ^{-{\f {11}4}-}$. Then, $S_2 L^2 = S^\prime_2L^2 \cap S^{\prime \prime} _2 L^2$. Moreover $\int yv(y)S_2\phi(y)dy=0$ and $PTS_2=QvG_1vS_2=S_2vG_1vQ=0$. 
Finally, $\phi=Uv\psi\in S_1 L^2$ belongs to $S_2L^2$ if and only if $\psi\in L^p$, $p>3$.  
\end{lemma}
\begin{proof}  It suffices to prove that $S_2 L^2 \subset S^\prime_2L^2 \cap S^{\prime \prime} _2 L^2$ since   reverse inclusion holds trivially.
Let $\phi \in S_1 L^2$. We have
\begin{align} \label{c_0} 
	 \la S_1T PT S_1\phi , \phi \ra = \la PT  \phi, PT  \phi \ra = \| PT  \phi \|_2^2. 
\end{align} On the other hand,  since $S_1v=0$ and $x,y$ and $v$ are real, we have
\begin{align} 
	\la S_1v G_1 v S_1\phi , \phi \ra =& \int_{\R^6} \phi (x) v (x) |x-y|^2  \overline{v(y) \phi(y)} dy dx \nn \\
    = & \int_{\R^6} \phi (x) v (x) [ |x|^2 - 2 x \cdot y - |y|^2 ] \overline{ v(y)\phi(y)} dy dx \nn\\ 
     = & -2\int_{\R^6} \phi (x) v (x)  x \cdot   \overline{ y v(y)\phi(y)} dy dx 
    = -2 \Big| \int_{\R^3} y v(y) \phi(y) dy \Big|^2 \label{S1vG1vS1} 
\end{align} 
Hence, if $\phi \in S_2 L^2$ then we  have 
$$
0 = \la T_1 \phi , \phi \ra =   \la T PT \phi , \phi \ra - \frac{\|V\|_1}{3 (8\pi)^2}  \la v G_1 v \phi , \phi \ra  =  \| PT  \phi \|_2^2+\frac{2\|V\|_1}{3 (8\pi)^2}  \Big| \int_{\R^3} y v(y) \phi(y) dy \Big|^2.
$$
Therefore,  
$$ 
 || PT \phi \|_2  =  \Big| \int_{\R^3} y v(y) \phi(y) dy \Big|=0,
$$ which yields the claim. 
 
This also implies that $\int yv(y)S_2\phi(y)dy=PTS_2=0$ and   
 $$
 QvG_1vS_2\phi=-2Q v(x)x\cdot \int y v(y)S_2\phi(y)dy=0.
 $$
 
Finally, by Lemma~\ref{lem:psiexp2}, $\psi\in L^p$, $3<p\leq \infty$ if and only if $PT \phi=\int yv(y) \phi(y)dy=0$, which is equivalent to $\phi \in S_2 L^2$ by the argument above.
\end{proof}

Define $S_3$ the projection on to the kernel of $T_2=S_2vG_3 vS_2+\frac{10}{3 }  S_2vWvS_2$, where $W(x,y)=|x|^2|y|^2$. Note that the kernel of $G_3$ is 
\begin{align}\label{exp4-1} 
|x-y|^4= |x|^4 + |y|^4 - 4x \cdot y |y|^2 - 4 y \cdot x |x|^2 + 2 |x|^2 |y|^2 + 4 ( x \cdot y )^2.  
\end{align}
Since $S_2x_jv=S_2v=0$,  all but the final two terms contribute zero to $S_2 vG_3v S_2$. Therefore the   kernel of $T_2$ (as an operator on $S_2L^2$) is
\be\label{T2ker}
T_2(x,y)=v(x)\big[\frac{26}{3 }|x|^2 |y|^2 + 4 ( x \cdot y )^2\big]v(y).
\ee

\begin{lemma}\label{lem:S_3} Let $|v(x)| \les \la x \ra ^{- 4-}$. Fix $\phi = Uv \psi \in S_2L^2$. Then $\phi\in S_3L^2 $ if and only if $\psi \in L^p $, for all $2 \leq p \leq \infty$. Moreover the kernel of 
$T_2$ agrees with the kernel of $S_2vG_3vS_2$.
\end{lemma}
\begin{proof} Using \eqref{T2ker} for $\phi\in S_2L^2$ , we have
$$
\la T_2\phi,\phi\ra = \frac{26}{3 }\bigg|\int_{\R^3} |y|^2 v(y) \phi(y)\, dy \bigg|^2 + 4 \sum_{i,j=1}^3 \bigg| \int y_i y_j v(y) \phi(y)\, dy \bigg| ^2.
$$
In particular, $T_2$ is positive semi-definite. Therefore  $\phi\in S_3L^2$, if and only if $\la T_2 \phi,\phi\ra=0$, which by the  calculation above equivalent to $\int y_i y_j v(y) \phi(y)\, dy=0$ for all $i, j$.
The claim now follows from Lemma~\ref{lem:psiexp2}.

The claim for $S_2 vG_3v S_2$ also follows from this since by the calculation before the lemma its  kernel is  $v(x)\big[2|x|^2 |y|^2 + 4 ( x \cdot y )^2\big]v(y)$.
 \end{proof}

\begin{lemma}  \label{invertible}Let $|v(x)| \les \la x \ra ^{- 4-}$. Then the kernel of the operator $S_3 v G_4 v S_3$ on $S_3L^2$ is trivial. 
\end{lemma} 
\begin{proof} Take $\phi$ in the kernel of $S_3 v G_4 v S_3$. Using \eqref{RH_0 rep}, we have (for $0<\lambda<1$)
$$
R (H_0;- \lambda^4)=\frac1{2i\lambda^2}\big[R_0(i\lambda^2)-R_0(-i\lambda^2)\big]= \frac{e^{i\sqrt i \lambda|x-y|} -e^{i\sqrt{-i} \lambda|x-y|} }{8\pi i \lambda^2|x-y|}.
$$
By an expansion similar to   \eqref{eq:R0low}, and the proof of Lemma~\ref{lem:M_exp}, we have for $0<\lambda<1$  and for all $|x-y|$, 
$$
	R (H_0; -\lambda^4)= \frac{a_0}{\lambda}  + G_0  + a_1  \lambda G_1 + a_3 \lambda^3 G_3 + a_4 \lambda^4 G_4 +  O (|\lambda|^{4+} |x-y|^{5+}),$$ 
where $a_0,a_1,a_3,a_4\in \C$ are constants. Notice that since $\phi \in S_3L^2$ one has $0 = \la v,\phi \ra = \la  G_1 v \phi, v\phi \ra=\la G_3 v \phi, v\phi \ra $. Also note that since $v\phi=V\psi$, we have 
$$
\iint |x-y|^{5+} v(x)v(y)|\phi(x)\phi(y)|dx dy\les \iint |x-y|^{5+} \la x\ra^{-8-} \la y\ra^{-8-} dx dy<\infty.
$$
Therefore  
\begin{align} \label{G5toG0}
	0&= \la S_3 v G_4 v \phi, \phi \ra = \la G_4 v \phi, v\phi \ra \\
	&=  \lim_{ \lambda \to 0 }  \Big\la \frac{R (H_0; - \lambda^4)- a_0 \lambda^{-1} - G_0 - a_1 \lambda G_1- a_3 \lambda^3 G_3}{\lambda^4} v\phi, v\phi\Big\ra \nn \\ 
	& = \lim_{ \lambda  \to 0  } \Big\la \frac{R (H_0; -\lambda^4)- G_0}{\lambda^4} v\phi, v\phi \Big\ra \nn.
\end{align}
Further, recalling that $G_0=[ \Delta^2]^{-1}$ and considering the Fourier domain, one has  
\begin{multline} \label{G5toG0lim}
	0= \lim_{ \lambda  \to 0 } \Big\la \frac{R(H_0;- \lambda^4)- G_0}{\lambda^4}v\phi, v\phi  \Big\ra \\ = \lim_{\lambda \to 0} \frac{1}{\lambda^4} \Big\la \Big( \frac{1}{ 8 \pi^2 \xi^4 + \lambda^4} - \frac{1}{ 8 \pi^2 \xi^4} \Big) \widehat{v\phi}(\xi), \widehat{v\phi}(\xi) \Big\ra \\ 
	= \lim_{ \lambda  \to 0} \int_{\R^3} \frac{-1}{(8 \pi^2 \xi^4 + \lambda^4) 8 \pi^2 \xi^4} |\widehat{v\phi}(\xi)|^2 d \xi = \frac{-1}{  64 \pi^4} \int_{\R^3} \frac{|\widehat{v\phi}(\xi)|^2 }{ \xi^8} d \xi.
\end{multline}
Where we used the Monotone Convergence Theorem in the last step.

Note that this gives $v\phi=0$ since $v\phi \in L^1$. Also noting that the support of $\phi=Uv\psi$ is a subset of the support of $v$, we have   $\phi=0$. This establishes the invertibility of  $S_3 v G_4 v S_3$ on $S_3L^2$.
\end{proof}
\begin{rmk} \label{G5toG0inuse}
Notice that, \eqref{G5toG0} and \eqref{G5toG0lim} imply that for any $\phi \in S_3$ one has 
\begin{align*}
\la S_3 v G_4 v \phi, \phi \ra = \frac{-1}{ 64 \pi^4} \int_{\R^3} \big\la \tfrac{\widehat{v\phi}(\xi) }{ \xi^4},  \tfrac{\widehat{v \phi}(\xi) }{ \xi^4} \big\ra= -\la G_0v \phi, G_0v \phi \ra
\end{align*}
provided $|v(x)| \les \la x \ra ^{- 4-}$.
\end{rmk}
\begin{lemma} The operator $ P_0:=G_0 v S_3 [ S_3 v G_4 v S_3]^{-1} S_3 v G_0$  is the orthogonal projection on $L^2$ onto the zero energy eigenspace of $H = \Delta^2 + V$. 
\end{lemma}

\begin{proof}
Let $ \{ \phi_k \}_{k=1}^N $ be the orthonormal basis of $S_3L^2$, then $S_3 f = \sum_{j=1}^N \phi_j \la f,  \phi_j \ra $. Moreover, for all $\phi_k$, one has $ \psi_k = - G_0 v \phi_k =-G_0V\psi_k$ are linearly independent for each $k$ and $\psi_k \in L^2$. We will show that $P_0 \psi_k  = G_0 v S_3 [ S_3 v G_4 v S_3]^{-1} S_3 v G_0 \psi_k = \psi_k$ for all $1\leq k \leq N$. This implies that $P_0$ is the identity on the range of $P_0$.
Since $P_0$ is self-adjoint, this finishes the proof.

Let $\{A_{ij}\}_{i,j=1}^N$ be the matrix that representation of $S_3 v G_4 v S_3$ with respect to  the orthonormal basis $ \{ \phi_k \}_{k=1}^N $, then by Remark~\ref{G5toG0inuse}
$$
A_{ij} =  \la S_3 v G_4 v \phi_j, \phi_i \ra = -\la  G_0 v \phi_j,  G_0 v \phi_i \ra =  -\la   \psi_j,    \psi_i \ra.
$$
Also note that, by the representation of $S_3$, we  have 
\be\label{s3vg0psik}
S_3 v G_0 \psi_k = \sum^N_{j=1} \phi_j \la v G_0 \psi_k,  \phi_j \ra = -\sum^N_{j=1} \phi_j \la \psi_k,  \psi_j \ra = - \sum^N_{j=1} \phi_j A_{j k}
\ee

By \eqref{s3vg0psik} we have 
\begin{multline*}
P_0 \psi_k =  -\sum^N_{j=1} G_0 v S_3 [ S_3 v G_4 v S_3]^{-1} \phi_j  A_{jk} \\  
 = -\sum^N_{i,j=1} G_0 v S_3 (A^{-1})_{ij}  \phi_i A_{jk} 
=  \sum^N_{i,j=1} \psi_i (A^{-1})_{i,j}  A_{jk}= \sum^N_{i =1} \psi_i \delta_{ik}=\psi_k.
\end{multline*}
\end{proof}
 
\begin{rmk}\label{rem:ort}
	
	One consequence of the preceeding results is that any zero-energy resonance function is of the form:
	$$
		\psi(x)= c_0+ c_1 \frac{x_1}{\la x \ra }+ c_2 \frac{x_2}{\la x \ra } +c_3 \frac{x_3}{\la x \ra } + \sum_{1\leq i \leq j\leq 3}  c_{ij} \frac{x_i x_j}{\la x \ra^{3}} + O_{L^2}(1).
	$$
	For some constants $c_0,c_1, c_2, c_3$, and $c_{ij}$, $1\leq i\leq j \leq 3$.
	Hence, the resonance space is at most 10 dimensional along with a finite-dimensional eigenspace. Moreover, $S_1-S_2 $ is at most four dimensional, $S_2-S_3$  is at most 6 dimensional, the rest is the eigenspace. 
\end{rmk}

\end{document}